\documentclass[a4paper, 11pt]{article}
\usepackage{authblk, setspace, subfig, caption}
\usepackage{amsmath}
\usepackage{amssymb,amscd, mathtools} 
\usepackage{enumerate} 
\usepackage[active]{srcltx} 
\usepackage{amsthm, leftidx}
\usepackage{multirow, booktabs}
\usepackage{algorithm}
\usepackage{algpseudocode}

\theoremstyle{definition}
\newtheorem{definition}{Definition}
\newtheorem{convention}{Convention}
\newtheorem{remark}{Remark}
\newtheorem{example}{Example}

\theoremstyle{plain}
\newtheorem{theorem}{Theorem}
\newtheorem{lemma}{Lemma}
\newtheorem{proposition}{Proposition}
\newtheorem{corollary}{Corollary}

\usepackage{tikz}
\usetikzlibrary{mindmap,shadows, calc}
\usepackage[hidelinks,pdfencoding=auto]{hyperref}

\usepackage{siunitx}
\newcommand{\RR}{\mathbb{R}}

\def\states{\mathcal X}

\def\cset{\mathcal M}
\def\csete{\mathcal C}
\def\allgambles{\RR^n}

\def\lpr{\underline P}
\def\upr{\overline P}

\def\gambleset{\mathcal F}

\usepackage[numbers]{natbib} 
\newcommand{\low}[1]{{\underline{#1}}}
\newcommand{\up}[1]{{\overline{#1}}}
\newcommand{\ncone}[2]{N(#1, #2)}
\newcommand{\fan}[1]{{\mathcal N(#1)}}
\newcommand{\charf}[1]{{\mathbb I_{#1}}}
\newcommand{\setcol}[1]{\mathcal #1}
\newcommand{\cone}[1]{\mathrm{cone}(#1)}
\newcommand{\ri}[1]{\mathrm{ri}(#1)}

\title{Normal cones corresponding to credal sets of lower probabilities}
\author[1]{Damjan \v{S}kulj \\ University of Ljubljana, Faculty of Social Sciences \\ Kardeljeva pl. 5, SI-1000 Ljubljana, Slovenia \\ \href{mailto:damjan.skulj@fdv.uni-lj.si}{\tt damjan.skulj@fdv.uni-lj.si} }

\begin{document}
	
	\maketitle

	\begin{abstract}
		Credal sets are one of the most important models for describing probabilistic uncertainty. They usually arise as convex sets of probabilistic models compatible with judgments provided in terms of coherent lower previsions or more specific models such as coherent lower probabilities or probability intervals. In finite spaces, credal sets usually take the form of convex polytopes. Many properties of convex polytopes can be derived from their normal cones, which form polyhedral complexes called normal fans. We analyze the properties of normal cones corresponding to credal sets of coherent lower probabilities. For two important classes of coherent lower probabilities, 2-monotone lower probabilities and probability intervals, we provide a detailed description of the normal fan structure. These structures are related to the structure of the extreme points of the credal sets. To arrive at our main results, we provide some general results on triangulated normal fans of convex polyhedra and their adjacency structure.
				
		\smallskip\noindent
		{\bfseries Keywords.} normal cone,  credal set, convex polyhedron, extreme point, imprecise probability, coherent lower probability  	
		
		\smallskip\noindent
		2020 Mathematics Subject Classification: 60A86, 52B11
	\end{abstract}

\section{Introduction}
The application of mathematical models involving probabilities often suffers from the lack of sufficient evidence to support a single model. Therefore, adherence to classical models requires unwarranted assumptions that lead to unreliable results. The lack of evidence in probabilistic models is often referred to as uncertainty or imprecision. While uncertainty can be understood as a general concept addressed by probabilistic models, imprecision is the term that explicitly describes situations that no particular probabilistic model can adequately describe. The theories of imprecise probabilities were developed to provide methods that can deal with such probabilistic models and produce the outputs where the imprecision is faithfully reflected. In most cases, probabilistic imprecision is described by sets of probability distributions that are consistent with the available information, rather than by a particular precise distribution. The sets are represented by various types of constraints, ranging from the most general coherent lower and upper previsions to more specific coherent lower and upper probabilities, probability intervals, $p$-boxes, belief and possibility functions, and other models.

In recent years, methods of imprecise probabilities \cite{augustin2014introduction, bradley2019imprecise, quaeghebeur2022introduction} have been applied to various areas of probabilistic modelling, such as stochastic processes \cite{decooman-2008-a, skulj:09IJAR}, game theory \cite{miranda2018shapley, nau2011imprecise}, reliability theory \cite{coolen2004use, khakzad2019system, oberguggenberger2009classical, utkin2007imprecise, yu2016comparing}, decision theory \cite{JANSEN2018112, montes2014decision, troffaes2007decision}, financial risk theory \cite{pelessoni2003convex, vicig2008financial}, computer science \cite{abellan2018increasing, quost2018classification, utkin2019imprecise, utkin2020imprecise}, copulas \cite{dolvzan2022some, omladivc2020constructing, omladivc2020final, omladivc2020full,  zhang2020quantification} and others. 

The multiple probabilistic models that make up an imprecise model usually form a set described by a finite number of constraints that may arise directly from the available information. When the constraints are linear, they yield closed and convex sets, called \emph{credal sets}. In fact, all the above models define convex credal sets. The convexity of the models is an advantage because they allow efficient computations by implementing linear programming techniques.

In addition to optimization with respect to credal sets, understanding the associated structures is often important to understanding the models. In the field of imprecise probabilities, the geometric approach is very common. It has been used particularly extensively in the theory of belief functions \cite{cuzzolin2010geometry, cuzzolin2020geometry}. It often involves the analysis of the extreme points of credal sets of different models \cite{antonucci2010credal, Miranda201544, miranda2003extreme, WALLNER2007339}. The nature and behaviour of convex sets in the extreme points and their neighbourhoods is important in dynamical systems, such as stochastic processes. These processes usually require the solution of multiple optimization problems, where the solution in earlier time steps determines the initial conditions for the later ones. This is important in discrete time models \cite{decooman-2008-a, skulj:09IJAR, vskulj2017perturbation}, and even more so in continuous time models \cite{de2017limit, krak2017imprecise, skulj:15AMC, skulj20uqop}. The latter models would in principle require optimization with respect to the credal set at every point in an interval, which is not feasible. Discretization methods are then used to compute approximate solutions \cite{erreygers2018imprecise, krak2017imprecise}. As an alternative to discretization, methods based on normal cones \cite{vskulj2020computing, skulj:15AMC, skulj20uqop} have been proposed. Their advantage over discretization methods is that if a solution remains within a single normal cone within a time interval, the process is linear within that interval. The application of linearity within normal cones is a consequence of a general principle of normal cone additivity, which we state explicitly in Proposition~\ref{prop-cone-additivity}. Another well-known example is the so-called comonotonic additivity of lower expectations with respect to 2-monotone lower probabilities \cite{den:97,sch:86}. Another application of normal cones in the theory of imprecise probabilities was proposed in \cite{vskulj2019errors}, where a method based on normal cones is developed to estimate the maximum distance between a credal set and its approximation based on a finite number of constraints.

The aim of the present article is to analyze the structure of complexes consisting of normal cones, also called normal fans, corresponding to credal sets generated by coherent lower probabilities. We propose several general results and then analyze in detail normal cones corresponding to 2-monotone lower probabilities and probability intervals. The normal cones approach provides a new characterization of the extreme points of the credal sets. In the case of 2-monotone lower probabilities, the characterization previously known from the literature is associated with normal cones. However, in the case of probability intervals, whose credal sets are more complex and diverse, our approach with normal cones proves to be very useful and allows a detailed analysis of their structure.

The article has the following structure. In the following section we give an overview of the underlying theory of convex polytopes and their normal cones. We are primarily concerned with minimal cones that may correspond to extreme points. These are characterized in abstract terms of relations between support vectors which do not depend on any particular realization of a convex polyhedron. In Section~\ref{ss-acgt}, the abstract structure of all possible minimal cones is endowed by an adjacency relation which allows a graph-theoretic interpretation. In Section~\ref{s-csip}, essential elements of models of imprecise probabilities are presented with special attention to finitely generated coherent lower previsions and their credal sets. In Section~\ref{s-nclp}, normal cones of credal sets of coherent lower probabilities are introduced and studied in detail, and the special case of 2-monotone lower probabilities is treated in Section~\ref{ss-nc2mon}. Finally, normal cones of credal sets corresponding to probability intervals are analyzed in Section~\ref{s-ncpri}.

\section{Normal cones}\label{s-nc}
\subsection{Normal cones and normal fans of convex polytopes}	
We first give some general elements from the theory of convex sets. Most of the notations and results are taken from \cite{bruns2009polytopes, gruber:07CDG, Gruenbaum2003, Rockafellar1970, Ziegler2012}. 

Let $\allgambles$ be a finite-dimensional vector space with the standard scalar product, which we denote by $xy$ or sometimes $x\cdot y$ for any pair of vectors $x, y\in \allgambles$. A \emph{polyhedron} in $\allgambles$ is an intersection of a finite number of half-spaces of the form $\{ x\in \allgambles \colon x f \geqslant b_f \}$, where $f\in \allgambles$ is a vector and $b_f$ is a constant. Thus, a convex polyhedron can be defined as
\begin{equation}\label{eq-convex-polyhedron}
	\csete = \{ x\in \allgambles \colon x f \geqslant b_f \text{ for all } f \in \gambleset \},
\end{equation}
where $\gambleset$ is a given finite collection of vectors and $\{ b_f\colon f\in \mathcal F\}$ is a collection of constants. A polyhedron that is bounded is called a \emph{(convex) polytope}. It is well known that every polytope in $\allgambles$ has an equivalent representation as a convex hull of finitely many extreme points (see, e.g., \cite[Theorem 14.2.]{gruber:07CDG}). A convex hull of $d+1$ affinely independent points in $\allgambles$ is called \emph{$d$-simplex} or simply \emph{simplex} and is a special case of a polytope. 

Some of the inequalities $x f\geqslant b_f$ in \eqref{eq-convex-polyhedron} may indeed be equalities. However, this case can be unified with the general case by replacing an equality condition $x f = b_f$ with two inequalities, $x f \geqslant b_f$ and $x (-f) \geqslant -b_f$. This allows us to use the simple description \eqref{eq-convex-polyhedron} throughout the text. 
The set $\csete'\subseteq \csete$ obtained by turning some (possibly none) of the inequalities $xf\ge b_f$ in \eqref{eq-convex-polyhedron} into equalities $xf = b_f$ is then a \emph{face} of $\csete$. Moreover, we consider the empty set as a face of any convex polyhedron, and clearly $\csete$ is also a face of itself. A face of a convex polyhedron is again a convex polyhedron of lower or equal dimension. A face of $\csete$ that is not equal to $\csete$ or $\emptyset$ is called a \emph{proper face}. The set of elements of $\csete$ not contained in any proper face is called a \emph{relative interior} of $\csete$ and is denoted by $\ri{\csete}$. 

Now let $\csete$ be a polyhedron and take a point $x\in \csete$ and define its \emph{normal cone} as the set 
\begin{equation}
	\ncone{\csete}{x} = \{ f\in \allgambles \colon x f \leqslant y f \text{ for every } y\in \csete \}.
\end{equation}
That is, the normal cone of $x$ is the set of all vectors $f$ for which $x = \arg\min_{y\in\csete}y f$. The minimum of the above expression is usually recognized as a linear programming problem where $\csete$ is the feasible set. Thus, the normal cone of $x$ can be understood as the set of all vectors $f$ such that the objective function $y f$ has an optimal solution in $x$. It is known that only points in the boundary region minimize objective functions and therefore only normal cones for these elements are nonempty. 

In the case of a bounded convex set, each objective function is minimized in at least one extreme point. In this case, the union of the normal cones of the extreme points is therefore the entire space $\allgambles$. The basic theory of normal cones can be found in most monographs on convex theory. A thorough analysis can also be found in \cite{lu2008normal}, in particular from the point of view of normal fans, which also play an important role in this article.

Throughout this article we will be concerned with cones generated as non-negative linear combinations of finite sets of vectors $\mathcal G$, denoted by $\cone{\mathcal G}$. Thus, 
\begin{equation}\label{eq-cone-generated}
	\cone{\mathcal G} = \left\{ \sum_{f\in \mathcal G} \alpha_f f \colon \alpha_f \ge 0 \text{ for all } f\in \mathcal G \right\}. 
\end{equation}
If a cone is of the form \eqref{eq-cone-generated}, we will simply say that it is \emph{generated} by $\mathcal G$. 
A cone consists of its faces and relative interior which is equal to the set of all strictly positive linear combinations of elements in $\mathcal G$:
\begin{equation}
	\ri{\cone{\mathcal G}} = \left\{ \sum_{f\in \mathcal G} \alpha_f f \colon \alpha_f > 0 \text{ for all } f\in \mathcal G \right\}. 
\end{equation} 

The following proposition holds (see \cite{gruber:07CDG}, Proposition 14.1).
\begin{proposition}\label{prop-cone-positive-hull}
	Let $\csete$ be a convex polyhedron represented in the form \eqref{eq-convex-polyhedron} and $x\in \csete$ a boundary point. Let $\mathcal G(x) = \{ f\in \gambleset\colon x f = b_f \}$. Then 
	\begin{equation}
		\ncone{\csete}{x} = \cone{\mathcal G(x)}. 
	\end{equation}
	Moreover, if $x$ is an extreme point of $\csete$, then $\dim \ncone{\csete}{x} = n$ (dimension of the vector space). 
\end{proposition}
A cone generated by a linearly independent set of vectors $\mathcal G$ is called \emph{simplicial} (see e.g. \cite{Ziegler2012}). Simplicial cones play a central role in this paper. In general, a normal cone may not be simplicial, and will therefore be subdivided into simplicial cones by means of triangulations. 

\begin{example}\label{ex-simple-cones}
	In Figure~\ref{fig-nc1}, a convex polytope $\mathcal K$ is presented in the form \eqref{eq-convex-polyhedron}, with $\gambleset = \{ f_1, \ldots, f_5\}$. Dashed lines represent the support lines of the form $\{x\colon xf_i = b_{f_i}\}$. Two extreme points $x$ and $y$ are depicted with the corresponding normal cones $\ncone{\mathcal K}{x} = \cone{\{f_1, f_5\}}$ and  $\ncone{\mathcal K}{y} = \cone{\{f_4, f_5\}}$. Both cones are simplicial. 
\end{example}
\begin{figure}
	\centering
	\begin{tikzpicture}[scale = 15]
		\coordinate (x) at (0.000,  0.000);
		\coordinate (z) at (0.577,  1.000);
		\coordinate (y) at (1.155,  0.000);
		
		
		\coordinate (E1) at (0.531,  0.280);
		\coordinate (E2) at  (0.531,  0.227);
		\coordinate (E3) at  (0.592,  0.191);
		\coordinate (E4) at (0.670,  0.202);
		\coordinate (E5) at (0.702,  0.477);
		
		\draw[gray, opacity=.8,  fill = gray, fill opacity = .3] (E1) -- (E2)--(E3)	--	(E4) --	(E5)  -- cycle;
		
		
		\node at (.62, .3) {\footnotesize $\mathcal K$}; 
		

		\coordinate (f12) at (0.100,  -0.000);
		\coordinate (f23) at (0.050,  0.087);
		\coordinate (f34) at (-0.014,  0.099);
		\coordinate (f45) at (-0.099,  0.011);
		\coordinate (f51) at (0.076,  -0.065);
		
		\draw[->, very thin] ($(E1)!.5!(E2)$) --  ++($-1*(f12)$) node[below] {\footnotesize $f_1$};
		\draw[->, very thin] ($(E2)!.5!(E3)$) --  ++($-1*(f23)$) node[right] {\footnotesize $f_2$};
		\draw[->, very thin] ($(E3)!.5!(E4)$) --  ++($-1*(f34)$) node[above right] {\footnotesize $f_3$};
		\draw[->, very thin] ($(E4)!.5!(E5)$) --  ++($-1*(f45)$) node[above] {\footnotesize $f_4$};
		\draw[->, very thin] ($(E5)!.5!(E1)$) --  ++($-1*(f51)$) node[above] {\footnotesize $f_5$};
		
		\draw[very thin, dashed] ($(E1)!-.8!(E2)$) -- ($(E1)!1.8!(E2)$);
		\draw[very thin, dashed] ($(E2)!-.8!(E3)$) -- ($(E2)!1.8!(E3)$);
		\draw[very thin, dashed] ($(E3)!-.8!(E4)$) -- ($(E3)!1.8!(E4)$);
		\draw[very thin, dashed] ($(E4)!-.3!(E5)$) -- ($(E4)!1.3!(E5)$);
		\draw[very thin, dashed] ($(E5)!-.3!(E1)$) -- ($(E5)!1.3!(E1)$);
		
		\fill[left color=gray!15!white, right color=gray!35!white!80!red, draw=white] (E1) --  ++($-1*(f12)$) -- ($(E1) - (f51)$) -- cycle;
		\draw[very thin, ->] (E1) --  ++($-1.1*(f12)$);
		\draw[very thin, ->] (E1) --  ++($-1.1*(f51)$);
		\node at ($(E1)+(-.1, .02)$) {\footnotesize $N(\mathcal K, x)$};
		
		\fill[top color=gray!15!white, bottom color=gray!35!white!80!red, draw=white] (E5) --  ++($-1.1*(f51)$) -- ($(E5) - 1.1*(f45)$) -- cycle;
		\draw[very thin, ->] (E5) --  ++($-1.1*(f51)$);
		\draw[very thin, ->] (E5) --  ++($-1.1*(f45)$);
		\node at ($(E5)+(.05, .02)$) {\footnotesize $N(\mathcal K, y)$};
		
		\node[right] at (E1) {\footnotesize $x$};
		\node[below right] at (E5) {\footnotesize $y$};
	\end{tikzpicture}
	\caption{Normal cones $\ncone{\mathcal K}{x} = \cone{\{f_1, f_5\}}$ and  $\ncone{\mathcal K}{y} = \cone{\{f_4, f_5\}}$.}\label{fig-nc1}
\end{figure}
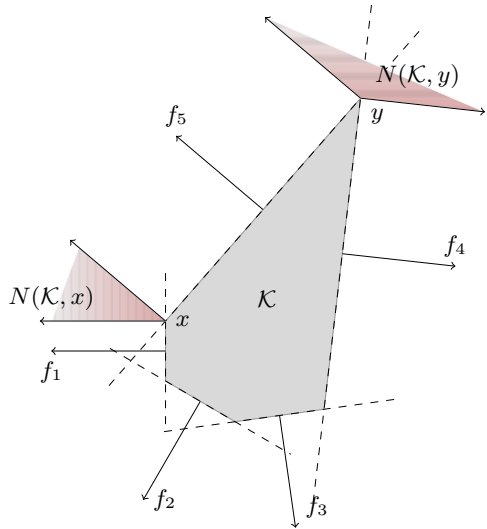
\begin{definition}[\cite{jensen2007non}]
	A collection $\mathcal P$ of polyhedra in $\allgambles$ is a \emph{polyhedral complex} (see Figure~\ref{fig-complex}) if 
	\begin{enumerate}[(i)]
		\item all proper faces of $P\in \mathcal P$ are in $\mathcal P$;
		\item an intersection of any two polyhedra in $\mathcal P$ is a face of both.  
	\end{enumerate}
\end{definition}
Let us consider some special cases of polyhedral complexes:
\begin{itemize}
	\item If $\mathcal P$ consists of cones, then we have a \emph{conical complex} or a \emph{fan}.
	\item If $\mathcal P$ consists of simplices, then it is a \emph{simplicial complex}.
	\item If $\mathcal P$ consists of simplicial cones, then we have a \emph{simplicial conical complex (simplicial fan)}. 
\end{itemize}
Despite similar notation, note that a simplicial conical complex consists of (simplicial) cones that are not simplices, and thus is not a simplicial complex. 
According to some definitions (see \cite{bruns2009polytopes}), polyhedral complexes as defined above are called \emph{embedded}, and fans are then embedded conical complexes. A special case of a fan is the collection of all normal cones of a polyhedron $\csete$, called \emph{normal fan} and denoted by $\fan{\csete} = \{ \ncone{\csete}{x}\colon x\in \csete \}$. Let $|\mathcal P|$ denote the \emph{union\footnote{The notation $|\mathcal P|$, which is standard in convex analysis, should not be confused with the usual notation for the cardinality of sets. In this paper it is used only in the case of polyhedral complexes to denote the union of polyhedra, while in all other cases it denotes the cardinality of sets.}} of all members of $\mathcal P$.  If $\mathcal P$ is a fan and  $|\mathcal P|=\allgambles$, then it is called a \emph{complete fan}. In the case of a polytope $\mathcal K$, every $f\in\allgambles$ belongs to at least one normal cone, and therefore, its normal fan is complete. 

The main focus of our analysis will be in simplicial fans. To transform a polyhedral complex into a simplicial one or a fan into a simplicial fan, the technique called \emph{triangulation} is used, as a special form of \emph{subdivision}. 
\begin{definition}[\cite{bruns2009polytopes}, Definition 1.45.]
	A \emph{subdivision} of a polyhedral complex $\mathcal P$ is another complex $\mathcal P'$ such that $|\mathcal P| = |\mathcal P'|$ and every polyhedron in $\mathcal P$ is a union of polyhedra of $\mathcal P'$. 
	
	If $\mathcal P$ and $\mathcal P'$ are polytopal, i.e., consisting of polytopes, or conical complexes and $\mathcal P'$ is simplicial, then it is called a \emph{triangulation} of $\mathcal P$ (see Figure~\ref{fig-complex})).
\end{definition}
\begin{figure}
	\begin{tikzpicture}[scale = 1.3]
		\coordinate (A) at (0,  0);
		\coordinate (B) at (1,  0);
		\coordinate (C) at (2,  -.5);
		\coordinate (D) at (3,  1.5);
		\coordinate (E) at (2,  1);
		\coordinate (F) at (1,  1);
		
		\draw[thick, fill = red!10!white] (A)--(B)--(E)--(F)--cycle;
		\draw[thick, fill = blue!10!white] (B)--(C)--(E)--cycle;
		\draw[thick, fill = red!10!white] (C)--(D)--(E)--cycle;
		
		\node at ($(A)!0.5!(E)$) {$P_1$};
		\node at ($(F)!0.6!(C)$) {$P_2$};
		\node at ($(B)!0.65!(D)-(0, .2)$) {$P_3$};
		
	\end{tikzpicture}
	\begin{tikzpicture}[scale = 1.3]
		\coordinate (A) at (0,  0);
		\coordinate (B) at (1,  0);
		\coordinate (C) at (2,  -.5);
		\coordinate (D) at (3,  1.5);
		\coordinate (E) at (2,  1);
		\coordinate (F) at (1,  1);
		
		\draw[thick, fill = red!10!white] (A)--(B)--(E)--(F)--cycle;
		\draw[red!20!white, very thin, fill = red!20!white] (A)--(B)--(F)--cycle;
		\draw[thick] (F)--(A)--(B);
		\draw[thick, fill = blue!10!white] (B)--(C)--(E)--cycle;
		\draw[thick, fill = red!10!white] (C)--(D)--(E)--cycle;
		
		\node at ($(A)!0.5!(E)$) {$P_1$};
		\node at ($(F)!0.6!(C)$) {$P_2$};
		\node at ($(B)!0.65!(D)-(0, .2)$) {$P_3$};
		
		\draw[dashed, thick] (B)--(F);
	\end{tikzpicture}
	\begin{tikzpicture}[scale = 1.3]
		\coordinate (A) at (0,  0);
		\coordinate (B) at (1,  0);
		\coordinate (C) at (2,  -.5);
		\coordinate (D) at (3,  1.5);
		\coordinate (E) at (2,  1);
		\coordinate (F) at (1,  1);

		\draw[thick, fill = red!10!white] (A)--(B)--(E)--(F)--cycle;
		\draw[red!20!white, very thin, fill = red!20!white] (A)--(E)--(B)--cycle;
		\draw[thick] (A)--(B);
		\draw[thick, fill = blue!10!white] (B)--(C)--(E)--cycle;
		\draw[thick, fill = red!10!white] (C)--(D)--(E)--cycle;
		
		\node at ($(A)!0.5!(E)$) {$P_1$};
		\node at ($(F)!0.6!(C)$) {$P_2$};
		\node at ($(B)!0.65!(D)-(0, .2)$) {$P_3$};
		
		\draw[thick, dashed] (A)--(E);
	\end{tikzpicture}
	\caption{A polyhedral complex (left) composed of three 2-dimensional polyhedra, their 1-dimensional faces (edges), and 0-dimensional faces (vertices). $P_2$ and $P_3$ are simplicial and $P_1$ is not. Two possible triangulations are shown in the middle and on the right.}\label{fig-complex}
\end{figure}
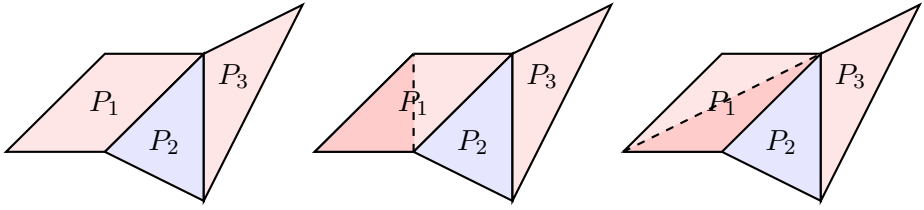
 It is known that every polytopal complex as well as every fan has a triangulation. More precisely, the following theorem holds for fans. 
\begin{theorem}[\cite{bruns2009polytopes}, Theorem 1.54]\label{thm-bruns}
	Let $\mathcal P$ be a conical complex and $\mathcal F\subset |\mathcal P|$ be a finite set of non-zero vectors such that $\mathcal F\cap C$ generates $C$ for every $C\in \mathcal P$. Then there exists a triangulation $\mathcal T$ of $\mathcal P$ such that $\cone{\{f\}}$, for $f\in \mathcal F$, are exactly the 1-dimensional faces of the elements of $\mathcal T$. 
\end{theorem}
In our case we will need triangulations whose 1-dimensional faces are exactly the elements of $\gambleset$, where $\gambleset$ is a set of vectors generating $\csete$ (see \eqref{eq-convex-polyhedron}). 
\begin{corollary}\label{cor-cone-triangulation}
	Let $\mathcal K$ be a polytope of the form \eqref{eq-convex-polyhedron}, where $\mathcal F$ is a given finite set of vectors in $\allgambles$. Then a conical complex $\mathcal T$ exists such that for every $C_T\in \mathcal T$ it holds that it is contained in a single normal cone $C\in \fan{\mathcal K}$ and is generated by a linearly independent set $\mathcal G\subset \mathcal F$. 
\end{corollary}
\begin{proof}
	Let $\fan{\mathcal K}$ be the normal fan of $\mathcal K$. By Proposition~\ref{prop-cone-positive-hull}, every $C\in \fan{\mathcal K}$ is generated by $\mathcal F\cap C$, whence by Theorem~\ref{thm-bruns}, a triangulation $\mathcal T$ of $\fan{\csete}$ exists such that $\cone{\{f\}}$ are exactly the 1-dimensional faces of $\mathcal T$. It follows by the definition of triangulation that every cone $C_T\in \mathcal T$ is a simplicial cone, whose 1-dimensional faces are of the form $\cone{\{f\}}$. Further, such a simplicial cone is then generated by the linearly independent set $\mathcal G = \{ f \colon \cone{\{f\}} \text{ is a face of } C_T\}$. Since $\mathcal T$ is a triangulation, every such $C_T$ is a subset of some normal cone $C\in \fan{\mathcal K}$. 
\end{proof}
A triangulation $\mathcal T$ of a normal fan $\fan{\mathcal K}$ is also a fan and will be called \emph{complete normal simplicial fan}. Note that the elements of a complete fan cover entire $\allgambles$. We require not only that $\mathcal T$ is a triangulation, but also that all elements of $\gambleset$ are its 1-dimensional faces. That is, no $f\in \gambleset$ lies in the relative interior of any $C_T\in \mathcal T$. Recall that by Corollary~\ref{cor-cone-triangulation} every cone $C_T\in \mathcal T$ is of the form $\cone{\mathcal G}$. A characterization of the elements of complete simplicial fans follows.
\begin{proposition}\label{prop-full-triangulation}
	Let $\cone{\mathcal G}$, for $\mathcal G\subset \gambleset$ be an element of a complete normal simplicial fan $\mathcal T$ obtained as a triangulation of $\fan{\mathcal K}$ for some polytope $\mathcal K$ of the form \eqref{eq-convex-polyhedron}, so that $\gambleset$ is the set of gambles forming the constraints. Then 
	\begin{enumerate}[(i)]
		\item $\mathcal G$ is linearly independent;
		\item for every $f\in \gambleset\backslash \mathcal G$ we have that $f\not\in \cone{\mathcal G}$;
		\item if $|\mathcal G| = \dim \allgambles$, then a convex set $\csete$ exists such that $\cone{\mathcal G}$ is its normal cone in an extreme point;
		\item a convex set $\csete$ exists such that $\cone{\mathcal G}$ is its normal cone.  
	\end{enumerate}
\end{proposition} 
\begin{proof}
	While (i) and (ii) are direct consequences of the definitions, we only prove (iii). Take some vector $x\in \allgambles$ and set $b_f = xf$ for every $f\in \mathcal G$ and $b_{f'} = xf' - 1$ for $f'\in \gambleset\backslash \mathcal G$. The set $\csete = \{ y \colon yf \ge b_f, y\in \allgambles, f\in \gambleset \}$ is clearly a convex subset of $\allgambles$. By the definition, $xf \ge b_f$ for all $f\in \gambleset$, whence $x\in \csete$. To see that $x$ is an extreme point, suppose it were a convex combination of two other points in $\csete$, say $u$ and $v$. By the construction, $uf = vf = b_f$ for every $f\in \mathcal G$. But the corresponding linear system has full rank and therefore $x$ is its unique solution, whence $u=v=x$. 
	Also by construction, $f\in \mathcal G$ are the only vectors in $\mathcal F$ that lie in $\ncone{\csete}{x}$, whence by Proposition~\ref{prop-cone-positive-hull}, $\ncone{\csete}{x} = \cone{\mathcal G}$, as claimed.
	 
	(iv). A linearly independent $\mathcal G$ can be completed to a basis $\mathcal G'$. By (iii), $\cone{\mathcal G}'$ is a normal cone in an extreme point. Hence, $\cone{\mathcal G}$ as its face is a normal cone as well. 
\end{proof}
\begin{definition}
	Let $\mathcal G\subseteq \gambleset$ be a basis of $\allgambles$, i.e., linearly independent with $|\mathcal G| = \dim \allgambles$. A cone of the form $\cone{\mathcal G}$, such that $f\not\in\cone{\mathcal G}$ for every $f\in\gambleset\backslash \mathcal G$, is called a \emph{maximal elementary simplicial cone (MESC)}. 
\end{definition}
\begin{example}
	Consider the polytope in Figure~\ref{fig-nc2}. Now $\gambleset$ contains an additional vector $f_6$, whose support line intersects those of $f_4$ and $f_5$ in $y$. Vector $f_6$ now lies in the relative interior of $\ncone{\mathcal K}{y} = \cone{\{f_4, f_5\}}$ and is therefore not a MESC. To meet our requirements, $\ncone{\mathcal K}{y}$ needs to be triangulated, that is subdivided in two cones. The only possible subdivision here is into $\cone{\{f_4, f_6\}}$ and $\cone{\{f_6, f_5\}}$, that both are MESCs. 
\end{example}
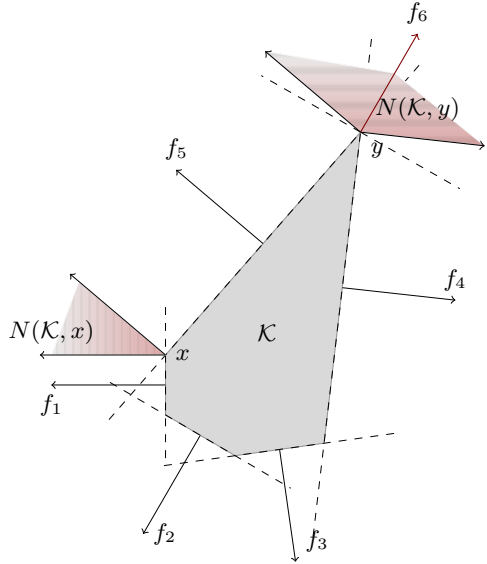
\begin{figure}
	\centering
		\begin{tikzpicture}[scale = 15]
		\coordinate (x) at (0.000,  0.000);
		\coordinate (z) at (0.577,  1.000);
		\coordinate (y) at (1.155,  0.000);
		
		
		\coordinate (E1) at (0.531,  0.280);
		\coordinate (E2) at  (0.531,  0.227);
		\coordinate (E3) at  (0.592,  0.191);
		\coordinate (E4) at (0.670,  0.202);
		\coordinate (E5) at (0.702,  0.477);
		
		\draw[gray, opacity=.8,  fill = gray, fill opacity = .3] (E1) -- (E2)--(E3)	--	(E4) --	(E5)  -- cycle;
		
		
		\node at (.62, .3) {\footnotesize $\mathcal K$}; 
		

		\coordinate (f12) at (0.100,  -0.000);
		\coordinate (f23) at (0.050,  0.087);
		\coordinate (f34) at (-0.014,  0.099);
		\coordinate (f45) at (-0.099,  0.011);
		\coordinate (f51) at (0.076,  -0.065);
		\coordinate (f65) at (-0.050,  -0.087);
		\coordinate (f65N) at (-0.087, 0.050);
		
		\draw[->, very thin] ($(E1)!.5!(E2)$) --  ++($-1*(f12)$) node[below] {\footnotesize $f_1$};
		\draw[->, very thin] ($(E2)!.5!(E3)$) --  ++($-1*(f23)$) node[right] {\footnotesize $f_2$};
		\draw[->, very thin] ($(E3)!.5!(E4)$) --  ++($-1*(f34)$) node[above right] {\footnotesize $f_3$};
		\draw[->, very thin] ($(E4)!.5!(E5)$) --  ++($-1*(f45)$) node[above] {\footnotesize $f_4$};
		\draw[->, very thin] ($(E5)!.5!(E1)$) --  ++($-1*(f51)$) node[above] {\footnotesize $f_5$};

		\draw[very thin, dashed] ($(E1)!-.8!(E2)$) -- ($(E1)!1.8!(E2)$);
		\draw[very thin, dashed] ($(E2)!-.8!(E3)$) -- ($(E2)!1.8!(E3)$);
		\draw[very thin, dashed] ($(E3)!-.8!(E4)$) -- ($(E3)!1.8!(E4)$);
		\draw[very thin, dashed] ($(E4)!-.3!(E5)$) -- ($(E4)!1.3!(E5)$);
		\draw[very thin, dashed] ($(E5)!-.3!(E1)$) -- ($(E5)!1.3!(E1)$);
		
		\fill[left color=gray!15!white, right color=gray!35!white!80!red, draw=white] (E1) --  ++($-1*(f12)$) -- ($(E1) - (f51)$) -- cycle;
		\draw[very thin, ->] (E1) --  ++($-1.1*(f12)$);
		\draw[very thin, ->] (E1) --  ++($-1.1*(f51)$);
		\node at ($(E1)+(-.1, .02)$) {\footnotesize $N(\mathcal K, x)$};
		
		\fill[top color=gray!15!white, bottom color=gray!35!white!80!red, draw=white] (E5) --  ++($-1.1*(f51)$) -- ($(E5) - 0.6*(f65)$) -- ($(E5) - 1.1*(f45)$) -- cycle;
		\draw[thin, ->] (E5) --  ++($-1.1*(f51)$);
		\draw[thin, ->] (E5) --  ++($-1.1*(f45)$);
		\node at ($(E5)+(.05, .02)$) {\footnotesize $N(\mathcal K, y)$};
		
		\node[right] at (E1) {\footnotesize $x$};
		\node[below right] at (E5) {\footnotesize $y$};
		
		\draw[->, thin, red!50!black] (E5) --  ++($-1*(f65)$) node[above, black] {\footnotesize $f_6$};
		\draw[very thin, dashed] (E5) --  ++($-1*(f65N)$);
		\draw[very thin, dashed] (E5) --  ++($(f65N)$);
	\end{tikzpicture}
	\caption{Vector $f_6$ is added to $\gambleset$ so that the corresponding support line intersects those of $f_4$ and $f_5$. $\cone{\{f_4, f_5\}}$ now contains $f_6$ in its relative interior and must therefore be triangulated into the union of MESCs $\cone{\{f_4, f_6\}}$ and $\cone{\{f_6, f_5\}}$.}\label{fig-nc2}
\end{figure}
As follows from Proposition~\ref{prop-full-triangulation}, maximal elements of complete normal simplicial fans of convex polytopes are exactly MESCs. However, this does not mean that every MESC is an element of a complete normal simplicial fan of a given polytope $\mathcal K$ generated by $\mathcal F$. There are two main reasons for this. The first reason is that a normal fan $\fan{\mathcal K}$ can have several different triangulations. The second, deeper reason is that different polyhedra in general have different normal fans, even if they are obtained by the same set of support vectors $\gambleset$. When two polyhedra have the same normal fan, they are called \emph{normally equivalent} (see \cite{de2010triangulations}). 

In the sequel, we will focus on the set of possible MESCs. We will analyze their structure in particular cases related to models of imprecise probabilities. Each complete normal simplicial fan contains MESCs together with their lower dimensional faces as building blocks. However, not all collections fit together. A useful tool for analyzing possible configurations of extreme points in polyhedra is the endowment of the graph structure to the set of its extreme points. In the next section, we build on this idea to generate a graph on the set of all MESCs that can reveal some structural properties of normal fans and, in particular, their triangulated forms.

\subsection{Adjacent cones and graph theoretical properties of maximal elementary simplicial cones}\label{ss-acgt}
A convex polytope $\mathcal K$ can be given a graph structure with the extreme points considered as vertices. An edge between two vertices is then a one-dimensional face of $\mathcal K$ with the given vertices as extreme points. In this section, we extend the graph structure to the set of MESCs. They are related to the extreme points in the sense of Proposition~\ref{prop-full-triangulation}. Any complete normal simplicial fan is a simplicial complex consisting of MESCs. The graph structure introduced in the sequel will give us some insights into the combination of MESCs to building the fans, and consequently into the possible polyhedral structures obtained in the form \eqref{eq-convex-polyhedron}. 

We begin with the notion of adjacency. Two extreme points of a convex polytope $\mathcal K$ are adjacent if they are connected by an edge. The normal cone corresponding to the edge, which is a one-dimensional face of $\mathcal K$, is a common face of the normal cones corresponding to the two extreme points. Moreover, the intersection of the two normal cones is exactly the common face of codimension 1, i.e., the dimension of the face is $n-1$, where $n$ is the dimension of the two cones, which means that their relative interiors are disjoint. We use this property as the definition of adjacency of two MESCs.
\begin{definition}
	Let $\mathcal G$ and $\mathcal G'$ be two subsets of $\gambleset$, so that the corresponding cones $\cone{\mathcal G}$ and $\cone{\mathcal G'}$ are MESCs. Then the cones are said to be adjacent if they intersect in a common face of codimension 1.  
\end{definition}
The following corollary is immediate. 
\begin{corollary}\label{cor-adjacent-mesc}
	Let cones $\cone{\mathcal G}$ and $\cone{\mathcal G'}$ be adjacent MESCs. Then $|\mathcal G\cap \mathcal G'| = |\mathcal G|-1 = |\mathcal G'|-1$. 
\end{corollary}
\begin{example}
	In Figure~\ref{fig-nc1}, cones $\ncone{\mathcal K}{x} = \cone{\{ f_1, f_5\}}$ and $\ncone{\mathcal K}{y} = \cone{\{ f_4, f_5\}}$ are adjacent MESCs, and $ \{ f_1, f_5\}\cap \{ f_4, f_5\} = \{ f_5\}$ is the generator set of the common face. 
\end{example}
However, the converse of the above corollary is not true. That is, if we have two sets $\mathcal G$ and $\mathcal G'$ that differ in one element, the corresponding cones are not necessarily adjacent. The reason is that their intersection can be larger than the common face. In Figure~\ref{fig-intersections}, two pairs of polytopes are shown, one of which has an intersection greater than the common face and the other intersects exactly in the common face. This situation can be easily applied to the case of cones by considering the cones whose intersections with a plane correspond to the polytopes shown. The following lemma contains the necessary and sufficient conditions for adjacency.
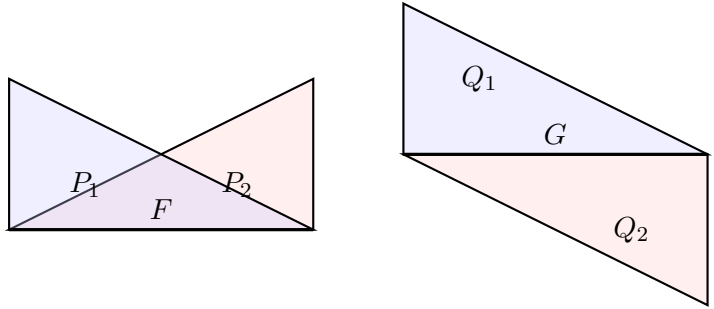
\begin{figure}
	\centering
	\begin{minipage}[c]{.4\textwidth}
	\begin{tikzpicture}[scale = 2]
		
		
		\coordinate (A) at (0,  0);
		\coordinate (B) at  (2,  0);
		\coordinate (C) at  (2, 1);
		\coordinate (D) at (0, 1);
		\coordinate (E) at (1, 0.5);
		
		\draw[thick, fill = red!20!white, fill opacity=0.30] (A)--(B)--(C)--cycle; 
		\draw[thick, fill = blue!20!white, fill opacity=0.30] (A)--(B)--(D)--cycle; 
		\draw[very thick] (A)--(B);
		
		\node[above] at (1, 0) {$F$};
		
		\node at (.5, .3) {$P_1$};
		\node at (1.5, .3) {$P_2$};
	\end{tikzpicture}
\end{minipage}
\begin{minipage}[c]{.4\textwidth}
\begin{tikzpicture}[scale = 2]
	
	
	\coordinate (A) at (0,  0);
	\coordinate (B) at  (2,  0);
	\coordinate (C) at  (2, -1);
	\coordinate (D) at (0, 1);
	\coordinate (E) at (1, 0.5);
	
	\draw[thick, fill = red!20!white, fill opacity=0.30] (A)--(B)--(C)--cycle; 
	\draw[thick, fill = blue!20!white, fill opacity=0.30] (A)--(B)--(D)--cycle; 
	\draw[very thick] (A)--(B);
	
	\node[above] at (1, 0) {$G$};
	
	\node at (.5, .5) {$Q_1$};
	\node at (1.5, -.5) {$Q_2$};
\end{tikzpicture}
\end{minipage}
	\caption{The intersection of $P_1$ and $P_2$ is larger than their common face $F$, while $Q_1$ and $Q_2$ intersect exactly in the common face $G$. }\label{fig-intersections}
\end{figure}
\begin{lemma}\label{lem-adjacent-cone}
	Let cones $\cone{\mathcal G}$ and $\cone{\mathcal G'}$ be MESCs such that $|\mathcal G\cap \mathcal G'| = |\mathcal G|-1 = |\mathcal G'|-1$. Let $\mathcal H$ be the hyperplane generated by $\mathcal G\cap \mathcal G'$ and $t$ its normal (non-zero) vector. Further let $f\in \mathcal G\backslash \mathcal G'$ and $f'\in \mathcal G'\backslash \mathcal G$. Then $\cone{\mathcal G}$ and $\cone{\mathcal G'}$ are adjacent if and only if $(f\cdot t)(f'\cdot t) < 0$, i.e. the scalar products with the normal vector have opposite signs.  
\end{lemma}
\begin{proof}
	Denote the vectors in $\mathcal G\cap \mathcal G'$ with $f_1, \ldots, f_{n-1}$ and without loss of generality we can assume that $\| t \| = 1$. By definition, $t\cdot f_i = 0$ for $i=1, \ldots, n-1$ and $\{f_1, \ldots, f_{n-1},t \}$ forms a basis of $\allgambles$. We can therefore write 
	\begin{align}
		f & = \sum_{i=1}^{n-1} \alpha_i f_i + \alpha t, \\
		f' & = \sum_{i=1}^{n-1} \beta_i f_i + \beta t.
	\end{align}
	We have that $f\cdot t = \alpha \| t \|^2 = \alpha$ and $f'\cdot t = \beta$. 
	
	We first prove that $\ri{\cone{\mathcal G}}\cap \ri{\cone{\mathcal G'}} \neq \emptyset$ implies that $\alpha\beta > 0$. Let $h\in \ri{\cone{\mathcal G}}\cap \ri{\cone{\mathcal G'}}$. Then we have that for some $\gamma_i>0, \delta_i>0$ for $i=1, \ldots, n-1$ and $\gamma, \delta > 0$,
	\begin{align}
		\label{eq-adj-t1}h & = \sum_{i=1}^{n-1} \gamma_i f_i + \gamma f = \sum_{i=1}^{n-1} \gamma_i f_i + \gamma \left(\sum_{i=1}^{n-1} \alpha_i f_i + \alpha t\right) = \sum_{i=1}^{n-1} \gamma'_i f_i + \gamma \alpha t 
		\intertext{and}
		\label{eq-adj-t2}h & = \sum_{i=1}^{n-1} \delta_i f_i + \delta f' = \sum_{i=1}^{n-1} \delta_i f_i + \delta \left(\sum_{i=1}^{n-1} \beta_i f_i + \beta t\right) = \sum_{i=1}^{n-1} \delta'_i f_i + \delta \beta t. 
	\end{align}
	By the uniqueness of linear combinations of the basis vectors, we obtain that $\gamma\alpha = \delta\beta$. Now, since $\gamma$ and $\delta$ are strictly positive, $\alpha$ and $\beta$ must be of equal sign. This completes the first part of the proof. 
	
	The second part of the proof is to show conversely that $\alpha \beta > 0$ implies $\ri{\cone{\mathcal G}}\cap \ri{\cone{\mathcal G'}} \neq \emptyset$. To see this, we only need to show that an $h$ of the form \eqref{eq-adj-t1} and \eqref{eq-adj-t2} exists. Suppose $\gamma_i > 0$ and $\gamma > 0$ are given. Then we set $\delta = \frac{\alpha}{\beta}\gamma$, which is of positive sign because of the equal sign of $\alpha$ and $\beta$. Now for every $i$ equating both expression gives:
	\begin{equation}
		\gamma'_i = \gamma_i + \gamma \alpha_i = \delta'_i = \delta_i + \delta \beta_i. 
	\end{equation}
	Solving the above equation for $\delta_i$ gives $\delta_i = \gamma_i + \gamma \alpha_i - \delta \beta_i$, which for sufficiently large $\gamma_i$ can be made positive for every $i=1, \ldots, n-1$. This concludes the proof of the proposition.  
\end{proof}  
\begin{corollary}\label{cor-adjacent-extreme-points}
	Let $\mathcal K$ be a polytope and $\mathcal T$ its complete normal simplicial fan. Let $C, C'\in \mathcal T$ be a pair of adjacent MESCs. Then exactly one of the following holds:
	\begin{enumerate}[(i)]
		\item An extreme point $x\in \mathcal K$ exists such that $C, C'\subseteq \ncone{\mathcal K}{x}$.
		\item A pair of extreme points $x, x'$ exists that lie in a common 1-dimensional face of $\mathcal T$ such that $C\subseteq \ncone{\mathcal K}{x}$ and $C'\subseteq \ncone{\mathcal K}{x'}$. 
	\end{enumerate} 
\end{corollary}
\begin{proof}
	Assume $\mathcal K$ in the form \eqref{eq-convex-polyhedron}. Cones $C$ and $C'$ are then of the form $\cone{\mathcal{G}}$ and $\cone{\mathcal{G}'}$ respectively where both $\mathcal G$ and $\mathcal G'$ are bases of $\allgambles$. Moreover, both $C$ and $C'$ lie within single normal cones -- not necessarily the same, that correspond to extreme points, say $x$ and $x'$. If $x=x'$, then (i) holds. Now assume that $x\neq x'$. By adjacency, we have that $|\mathcal G \cap \mathcal G'| = n-1$. Normal cone containing $\cone{\mathcal G \cap \mathcal G'}$ corresponds to a face $F$ of $\mathcal K$ satisfying $yf = b_f$ for every $f\in \mathcal G \cap \mathcal G'$ and every $y\in F$. The face $F$ is therefore a subset of the set of solutions of $n-1$ independent linear equations, which is a 1-dimensional subspace of $\allgambles$. Consequently, $F$ must be an at most 1-dimensional face containing $x$ and $x'$, which are both particular solutions of the same set of linear equations. As assumed, $x\neq x'$, which implies that $F$ is exactly one dimensional face containing both points. 
\end{proof}
\begin{example}
	\begin{enumerate}[(i)]
		\item Consider the polytope $\mathcal K$ in Figure~\ref{fig-nc1}. The cones $\ncone{\mathcal K}{x}$ and $\ncone{\mathcal K}{y}$ are adjacent, meeting in $\cone{\{f_5\}}$ that is their common face. Notice that they correspond to adjacent extreme points $x$ and $y$.
		\item Cones $\cone{\{f_4, f_6\}}$ and $\cone{\{f_6, f_5\}}$ in Figure~\ref{fig-nc2} are adjacent as well, yet they correspond to the same extreme point $y$. 
	\end{enumerate}
	
\end{example}
With $\mathbb{F}$ we now denote the set of all MESCs. This set represents all possible maximal simplicial cones of polyhedra in the form \eqref{eq-convex-polyhedron}. Let $G = (\mathbb{F}, \mathbb{E})$ be the graph with the set of vertices $\mathbb{F}$ and the set of edges $\mathbb{E} = \{ (C, C')\colon C, C'\in \mathbb{F}, C \text{ and } C' \text{ are adjacent}\}$.  
\begin{proposition}
	Let $\mathcal T$ be a complete normal simplicial fan of a polytope of the form \eqref{eq-convex-polyhedron} and let $\mathbb{T}$ denote the set of its MESCs. Then the subgraph $G(\mathcal T)$ of $G(\mathbb{F}, \mathbb{E})$ with vertices $\mathbb{T}$ is connected and $n$-regular. 
\end{proposition}
\begin{proof}
	To see that the graph is $n$-regular, recall that $|\mathbb{T}| = \allgambles$. Every cone $C_T\in \mathbb{T}$ is simplicial and thus has exactly $n$ faces of codimension 1. As there are no border faces, every such face is then also a part of another, adjacent cone. 
	
	It is also clear that every two simplicial cones are connected by a chain of cones where each one has a common face of codimension 1 with the neighbours in the chain. 
\end{proof}
Every MESC $\cone{\mathcal G}$ defines a unique vector $x$ satisfying $xf = b_f$ for every $f\in \mathcal G$. Not every such vector however is an extreme point in $\mathcal K$. A characterization follows. 
\begin{proposition}
	A MESC $\cone{\mathcal G}$ lies within the normal cone of an extreme point $x\in \mathcal K$ of the form \eqref{eq-convex-polyhedron} if and only if the following conditions are satisfied:
	\begin{enumerate}[(i)]
		\item $xf = b_f$ for every $f\in \mathcal G$;
		\item $xf \ge b_f$ for every $f\in \mathcal F$.
	\end{enumerate}
\end{proposition}
\begin{proof}
	Condition (ii) implies that $x\in \mathcal K$, while condition (i) implies that it is a face of dimension 0. The dimension follows by linear independence of $\mathcal G$. Hence, $x$ is an extreme point in $\mathcal K$.  
\end{proof}
Now let $\mathcal K$ be given and let $\cone{\mathcal G}$ be a MESC within a normal cone $\ncone{\mathcal K}{x}$. The set $\mathbb{F}$ contains all MESCs corresponding to extreme points in $\mathcal K$, but not every cone in $\mathbb{F}$ corresponds to an extreme point. Additionally, we know that those cones in $\mathbb{F}$ that are adjacent to $\cone{\mathcal G}$ are exactly those corresponding to the same extreme point $x$ or an adjacent one, as follows by Corollary~\ref{cor-adjacent-extreme-points}. Moreover, every adjacent MESC is of the form $\cone{\mathcal G'}$, where, by Corollary~\ref{cor-adjacent-mesc}, $\mathcal G' = (\mathcal G\backslash\{ f\})\cup \{f'\}$. The following proposition shows that given $f\in \mathcal G$ adjacent cones must correspond to the same adjacent extreme point. 
\begin{proposition}
	Let MESC $\cone{\mathcal G}$ lie within a normal cone $\ncone{\mathcal K}{x}$ and take some $f\in \mathcal G$. Now suppose that for some $f'$ and $f''\in \mathcal F$ both cones $\mathcal G' = (\mathcal G\backslash\{ f\})\cup \{f'\}$ and $\mathcal G'' = (\mathcal G\backslash\{ f\})\cup \{f''\}$ are MESC and they lie within normal cones $\ncone{\mathcal K}{x'}$ and $\ncone{\mathcal K}{x''}$. Then $x' = x''$. (Possibly, $x'=x'' = x$.)
\end{proposition}
\begin{proof}
	By the assumptions, both $\cone{\mathcal G'}$ and $\cone{\mathcal G''}$ are adjacent to $\cone{\mathcal G}$. Let $t$ be the normal vector to the hyperplane $\mathrm{lin}(G\backslash\{ f\})$. By Lemma~\ref{lem-adjacent-cone}, $(f'\cdot t)(f''\cdot t) > 0$, whence the relative interiors of the corresponding cones intersect. Hence, the relative interiors of $\ncone{\mathcal K}{x'}$ and $\ncone{\mathcal K}{x''}$ intersect, which is only possible if the cones are the same.
\end{proof}
\begin{remark}
	It is indeed possible that we have multiple adjacent cones to $\cone{\mathcal G}$ corresponding to the same adjacent extreme point, as they might correspond to different triangulations of $\ncone{\mathcal K}{x'}$. Yet, when a triangulation is fixed, only one adjacent cone of the form $\cone{(\mathcal G\backslash\{ f\})\cup \{f'\}}$ exists per fixed $f$. 
\end{remark}
The above propositions would therefore enable identification of all cones corresponding to a complete normal simplicial fan of a convex polyhedron, and consequently an identification and enumeration of its extreme points. The main steps are given by Algorithm~\ref{alg-extreme-points}. 
\begin{algorithm}
	\caption{Algorithm: find complete simplicial fan and extreme points}\label{alg-extreme-points}
	\begin{algorithmic}[1]
		\Function{SimplicialFan}{$\mathcal F, \mathbf b$} 
		\State \Comment{Inputs: support vector and support function}
		\State {Initialize: find $\mathcal G$, such that $\cone{\mathcal G}$ is MESC}
		\State $V \gets \{ \mathcal G\}, E \gets \emptyset, G \gets (V, E)$
		\State \Comment{Adjacency graph of the members of complete simplicial fan}
		\State $\mathcal E \gets \mathrm{ExtremePoint}(\mathcal G)$
		\State \Comment{Collection of extreme points}
		\Repeat 
			\State {Select $\mathcal G\in V$ such that $\deg(\mathcal G) < n$}
			\State {neighbours $\gets$ Find all neighbour cones corresponding to extreme points}
			\State {neighbours $\gets$ Remove all duplicate neighbours $\mathcal G'$ having the same intersection with $\mathcal G$}
			\State $V \gets V \cup$ neighbours
			\State $E \gets E \cup \{ (\mathcal G, \mathcal G') \}$ for all $\mathcal G'\in$ neighbours
			\State $\mathcal E \gets \mathcal E \cup \{\mathrm{ExtremePoint}(\mathcal G')\}$ for all $\mathcal G'\in$ neighbours
		\Until ($G$ is $n$-regular)
		\EndFunction
	\end{algorithmic}
\end{algorithm}

\section{Credal sets as imprecise probability models}\label{s-csip}
In this section we introduce the basic concepts of imprecise probabilities used in this paper. When possible, we will stick to the standard terminology used in the theory of imprecise probabilities (see \cite{augustin2014introduction, walley:91}). The main goal of this section is to relate the models of imprecise probabilities to the concepts of convex analysis. 

The object of our analysis are \emph{ coherent lower previsions} \cite{miranda2008, TroffaesDeCooman2014}, which represent one of the most general models of imprecise probabilities. They include several special models, such as \emph{coherent lower and upper probabilities} \cite{dempster2008upper, pearl1988probability, weic:01}, 2- and n-\emph{monotonic capacities} \cite{bronevichAugustin2009, montes20182, sundberg1992characterizations}, \emph{belief} and \emph{plausibility functions} \cite{shafer1987belief, smets1993belief}, and others. Mathematically, coherent lower previsions are superlinear functionals that can equivalently be represented as lower envelopes of expectation functionals.

Let $\states$ represent a finite set -- \emph{sample space}, and $\allgambles$ the set of all real-valued maps on $\states$ -- \emph{gambles}. By $\charf{A}$ we will denote the \emph{indicator gamble} of a set $A\subseteq \states$:
\begin{equation}
	\charf{A}(x) = 
	\begin{cases}
		1 & x \in A, \\
		0 & \text{otherwise}. 
	\end{cases} 	
\end{equation}
A \emph{linear prevision} $P$ is an expectation functional with respect to some probability mass vector $p$ on $\states$. It maps a gamble $f$ into a real number $P(f)$:
\begin{equation}
	P(f) = \sum_{x\in\states} p(x)f(x).
\end{equation} 
Equivalently, we can write $P(f) = pf$. 
The set of linear previsions is therefore a subset of $\allgambles$. 

A \emph{lower prevision} is a mapping $\lpr\colon \gambleset\to \RR$. In this paper we will require lower previsions to satisfy the property of \emph{coherence}. We first define coherence for lower previsions with the domain $\allgambles$. A mapping $\lpr \colon \allgambles \to \RR$ is called a \emph{coherent lower prevision} if and only if it satisfies the following axioms (\cite{miranda2008}) for all $f, g\in \gambleset$ and $\lambda \ge 0$: 
\begin{itemize}
	\item[(P1)] $\lpr (f) \ge \inf_{x\in \states} f(x)$ [accepting sure gains];
	\item[(P2)] $\lpr (\lambda f) = \lambda \lpr (f)$ [positive homogeneity];
	\item[(P3)] $\lpr (f + g) \ge \lpr(f) + \lpr(g)$ [superlinearity].
\end{itemize}
A lower prevision with a restricted domain $\gambleset$ is said to be coherent whenever it can be extended to a coherent lower prevision on $\allgambles$. Equivalently, a lower prevision $\lpr$ on $\gambleset$ is coherent if and only if it allows the representation 
\begin{equation}
	\lpr (f) = \min_{P\in\cset} P(f)
\end{equation}
for every $f\in \gambleset$, where $\cset$ is a closed and convex set of linear previsions. The largest such set, denoted $\cset(\lpr)$, is called the \emph{credal set} of $\low P$, defined as
\begin{equation}
	\cset(\lpr) = \{ P \text{ a linear prevision }\colon P(f) \ge \lpr (f) \text{ for every } f\in \gambleset \}.
\end{equation}
Given a coherent lower prevision $\lpr$ on $\gambleset$, it is possible to extend it to the set of all gambles $\allgambles$ in possibly several different ways. However, there is a unique minimal extension, called the \emph{natural extension}:
\begin{equation}
	\low E(f) = \min_{P\in\cset(\lpr)} P(f). 
\end{equation} 
As the natural extension is the \emph{lower envelope} or the \emph{support function} of a credal set, containing expectation functionals, we may call a coherent lower prevision defined on the entire $\allgambles$ a \emph{lower expectation functional}. 
Together with a coherent lower prevision, the notion of \emph{coherent upper prevision} $\upr$ is often introduced. Assuming the domain is a vector space, the conjugacy relation $\lpr(f) = -\upr(-f)$ is a straightforward consequence of coherence. 

\subsection{Lower probabilities and probability intervals}
From the definition of a coherent lower prevision it follows that its domain $\gambleset$ may be a finite subset of $\allgambles$. Indeed, $\lpr (f)$ are often some judgments representing the available information about the expectations of the gambles. According to the philosophy of imprecise probabilities, probability models should not assume more information than is available. Therefore, in general, several probabilistic models fit the given information. The natural extension then allows all compatible models to be considered simultaneously, which is considered the main advantage of imprecise models over classical models. 

Lower previsions whose domain $\gambleset$ consists of indicator functions $\charf{A}$ of subsets $A\subseteq \states$ are called \emph{lower probabilities}. A lower probability is usually denoted by $L\colon \mathcal A\to \RR$, where $\mathcal A\subseteq 2^\states$. The mapping $L$ has the same role as $\lpr$ in the previous section. The credal set of a lower probability $\cset(L)$ is again the set of all linear previsions $\lpr \colon \gambleset \to \RR$ that satisfy $P(A)\ge L(A)$ for each $A\in\mathcal A$. A lower probability is \emph{coherent} if all bounds are reachable: $L(A) = \min_{P\in\mathcal{M}(L)}P(A)$. All lower probabilities in this paper are assumed to be coherent.

In the case where $\mathcal A\neq 2^\states$, the corresponding lower probability is said to be partially specified. An important example of a partially specified lower probability is \emph{probability interval model (PRI)} \cite{decampos:94, weic:01}, where the domain $\mathcal A$ contains only singletons and their complements. A more conventional way to introduce a probability interval is to have a pair of mappings $l, u\colon \states \to \RR$, such that $l\le u$. Here $l(x)$ is interpreted as the lower probability of $\{ x\}$ and $u(x)$ as the upper probability. The notion of coherence applies analogously to this case and is assumed throughout this paper unless explicitly stated otherwise. The conjugacy relation $u(x) = \upr(\charf{\{x \}}) = 1-\lpr(-\charf{\{x \}}) = 1-L(\{x\}^c)$, which implies $L(\{x\}^c) = 1-u(x)$, now allows us to regard probability intervals as lower probabilities with the domain $\mathcal A = \{ \{x \}\colon x\in\states\} \cup \{ \{x \}^c\colon x\in\states\}$. In general, we will denote a PRI model by an ordered pair $(l, u)$, and interpret it as a partially specified lower probability $L$ with the domain $\mathcal A$ if needed. 
Credal sets corresponding to lower probabilities and probability interval models will be denoted by $\cset(L)$ and $\cset(l, u)$, respectively.

\subsection{Finitely generated credal sets as polyhedra} 
A credal set is a closed and convex set of linear previsions. Since every linear prevision can be uniquely represented as a probability mass vector, a credal set can be represented as a convex set of probability mass vectors. The set $\cset$ is therefore the maximal set of $n$-dimensional vectors $p$ satisfying 
\begin{align}
	p f & \ge \lpr(f) & & \text{ for every } f\in \gambleset, \label{eq-cset-constraints-first}\\
	p \charf{\{x \}} & \ge 0 & & \text{ for every } x\in \states\label{eq-cset-constraints-nonnegativity} \text{ and } \\
	p \charf{\states} & = 1. \label{eq-constraint-equality} 	
\end{align}	
In our case, $\gambleset$ is assumed finite, and since additionally the set of all linear previsions is bounded, the corresponding credal set is therefore a convex polytope. A credal set that is a convex polytope is called a \emph{finitely generated credal set}, and a coherent lower prevision defined on a finite set of gambles is a \emph{finitely generated coherent lower prevision}.

According to the above, it would be suitable to extend the domain of $\lpr$ with the gambles of the form $\charf{\{x \}}$ for every $x\in\states$ in order to avoid a separate set of constraints \eqref{eq-cset-constraints-nonnegativity}. Doing so, however, may result in a non-coherent lower prevision, because other constraints may already imply that $\lpr(\charf{\{x \}}) \ge 0$, where the inequality may even be strict. Therefore we adopt the following convention: 
\begin{convention}\label{conv-include-1x}
	The domain $\gambleset$ of all coherent lower previsions used will contain all gambles of the form $\charf{\{x \}}$ for $x\in \states$, together with the value $\lpr(\charf{\{x \}})=0$, unless $\lpr(\charf{\{x \}})\ge 0$ is already implied by other values of $\lpr$ on $\gambleset$. 
\end{convention}
Assuming the above convention, the credal set of a coherent lower prevision $\cset(\lpr)$ is the set of vectors $p$ satisfying constraints \eqref{eq-cset-constraints-first} and \eqref{eq-constraint-equality}. 

Normal cones corresponding to credal sets of coherent lower previsions will be equivalently denoted by $\ncone{\lpr}{P} = \ncone{\cset(\lpr)}{P} = \{ f\colon P(f) = \lpr(f) \} = \cone{\gambleset(P)}$, where $\gambleset(P) = \{ f\in \gambleset\colon P(f) = \lpr(f) \}$. In addition to inequality constraints, credal sets satisfy the equality constraint \eqref{eq-constraint-equality}. This implies that every $\gambleset(P)$ contains the constant gamble $\charf{\states}$. Hence, every normal cone is of the form 
\begin{align}\label{eq-nc-form}
	\ncone{\lpr}{P} & = \cone{\gambleset(P)\backslash \{ \charf{\states}\}} + \mathrm{lin} \{ \charf{\states}\} \\
	& = \left\{ \sum_{f\in \gambleset(P)\backslash \{ \charf{\states}\}} \alpha_f f+ \beta \charf{\states}\colon \alpha_f \ge 0, \beta \in \RR \right\}.
\end{align}
Note that $\mathrm{lin}(\mathcal G)$ stands for the linear space spanned by $\mathcal{G}$, i.e. the set of all linear combinations of its elements.
Technically, a normal cone as above is a union of two symmetrical cones, one with $\beta \ge 0$ and the other with $\beta \le 0$.  The following result is straightforward.
\begin{proposition}[Normal cone additivity]\label{prop-cone-additivity}
	Take arbitrary vectors $g, h\in \ncone{\lpr}{P}$. Then $\lpr(g+h) = \lpr (g) + \lpr(h)$. 
\end{proposition}
\begin{proof}
	The fact that $g$ and $h$ both belong to the same normal cone at $P$ implies that $P(g) = \lpr(g)$ and $P(h)=\lpr(h)$. By the closure for sums of the cone, we then also have that $P(g+h) = \lpr(g+h)$ and therefore additivity of $P$ implies $\lpr(g+h) = \lpr (g) + \lpr(h)$. 
\end{proof}
This proposition can be understood as a generalization of the well known property of comonotone additivity for 2-monotone lower probabilities. As we show in Section~\ref{ss-nc2mon}, pairwise comonotone vectors always lie in common normal cones of 2-monotone lower probabilities. 

\section{Normal cones of coherent  lower probabilities}\label{s-nclp}
By a normal cone corresponding to a lower probability or a PRI model, we mean a normal cone corresponding to its credal set, and denote it by $\ncone{L}{P}$ or $\ncone{(l, u)}{P}$ respectively. Such a cone is then a non-negative hull of a set of indicator functions corresponding to a collection of subsets of $\states$.  This set includes $\charf{\states}$, as follows from \eqref{eq-nc-form}. Such a normal cone can thus be represented with the corresponding collection of sets. Let $\setcol{A}$ be arbitrary (finite) collection of sets. Denote the cone generated by their indicator functions by 
\begin{equation}\label{eq-cone-construction}
	\cone{\setcol{A}} := \cone{\{ \charf{A}\colon A\in \setcol{A}\backslash\{\states\}\}} + \mathrm{lin} \{ \charf{\states}\}.
\end{equation}
In particular, we characterize elements of complete normal simplicial fans. The following proposition holds. 
\begin{proposition}\label{prop-lprob-triangulation}
	Let $\cone{\setcol{A}}$ correspond to an element of a complete normal simplicial fan. Then:  
	\begin{enumerate}[(i)]
		\item The vectors $\{ \charf{A}\colon A\in \setcol{A} \}$ are linearly independent.
		\item No equation of the form $\sum_{\setcol{A}\backslash\{\states\}}\alpha_A \charf{A}  + \alpha_\states \charf{\states} = \charf{B}$, where $B\not\in \setcol{A}$, has a solution such that $\alpha_A\ge 0$ for every $A\in \setcol{A}\backslash\{\states\}$. 	
	\end{enumerate}	
\end{proposition}
\begin{proof}
	The proposition is an easy consequence of Proposition~\ref{prop-full-triangulation}: (i) follows directly from  Proposition~\ref{prop-full-triangulation}(i); (ii) states that no $\charf{B}$ can belong to $\cone{\setcol{A}}$ if $B\not \in \mathcal A$, which is required by Proposition~\ref{prop-full-triangulation}(ii) for general polytopes. 
\end{proof}
\begin{corollary}\label{cor-non-disjoint}
	Let $\cone{\setcol{A}}$ be an element of a complete normal simplicial fan corresponding to a credal set of a coherent lower probability $L\colon 2^\states\to \RR$. Then 
	\begin{enumerate}[(i)]
		\item for every pair of sets $A_1, A_2\in \mathcal A, A_1\cap A_2 \neq \emptyset$, i.e. $\mathcal A$ is an intersecting collection;
		\item for every pair of sets $A_1, A_2\in \mathcal A\backslash\{\states\}, A\cup B \neq \states$. 
	\end{enumerate}	 
\end{corollary}
\begin{proof}
	To see (i), suppose $A_1, A_2\in \setcol{A}$ exist such that $A_1\cap A_2 = \emptyset$. Then $\charf{A_1} + \charf{A_2} = \charf{A_1\cup A_2}$. Hence, either $A_1\cup A_2\in \setcol{A}$, which makes it linearly dependent or $\charf{A_1\cup A_2}\in \cone{\setcol{A}}$, violating (ii) of Proposition~\ref{prop-lprob-triangulation}. 
	
 To see (ii), suppose $A_1\cup A_2 = \states$. Then, $\charf{A_1}+\charf{A_2} = \charf{A_1\cup A_2} + \charf{A_1\cap A_2} = \charf{\states} + \charf{A_1\cap A_2}$. Hence, either $A_1\cap A_2 =\emptyset$, which is violation of (i) or taking $B = A_1\cap A_2$, $\charf{B}$ is contained in $\cone{\mathcal A}$. Then either, $\mathcal A$ is dependent, which violates Proposition~\ref{prop-lprob-triangulation}(i) or it contains $\charf{B}$, violating Proposition~\ref{prop-lprob-triangulation}(ii).
\end{proof}
Even though Corollary~\ref{cor-non-disjoint} turns out useful to characterize elements of complete simplicial fans, in general, it is insufficient, as the following example shows.
\begin{example}\label{ex-ex1}
	Let $\states = \{ x_1, x_2, x_3, x_4 \}$. Consider the collection of sets $A_1 = \{ x_1, x_2\}, A_2 = \{ x_2, x_3\}, A_3 = \{ x_1, x_3\}, \states$. It is an intersecting collection, yet it does not generate an element of a complete simplicial fan. This is because $\frac{1}{2}(\charf{A_1} + \charf{A_2}+ \charf{A_3})  = \charf{\{x_1, x_2, x_3\}}$, which violates Proposition~\ref{prop-lprob-triangulation} (ii). 
\end{example}
\begin{example}\label{ex-ex2}
	Let $\states = \{ x_1, x_2, x_3 \}$. Every MESC is a cone of the form $\cone{\mathcal A}$, where $\mathcal A$ is a collection of exactly 3 subsets of $\states$, including $\states$. Suppose $\{ x_1,  x_2 \}\in \mathcal A$. According to Corollary~\ref{cor-non-disjoint}, the only possibilities for the second set (note that $\states$ is always a member) in $\mathcal A$ are $\{ x_1\}$ and $\{ x_2\}$. Indeed, both these sets correspond to adjacent MESCs. Similarly, in addition to a singleton set $\{ x_1\}$, the only possible sets are $\{ x_1, x_2\}$ and $\{ x_1, x_3\}$, that both correspond to two adjacent MESCs corresponding to adjacent extreme points. 
\end{example}

\subsection{Normal cones of 2-monotone lower probabilities}\label{ss-nc2mon}
A lower probability $L\colon 2^\states\to \RR$ is said to be \emph{2-monotone (convex, supermodular)} if for every pair of sets $A, B\in 2^\states$, 
\begin{equation}\label{eq-2-monotone}
	L (A\cup B) + L (A\cap B) \ge L(A) + L(B).
\end{equation} 
A collection of subsets $\mathcal A$ is a \emph{chain} if for any pair of sets $A, B\in \mathcal A$, at least one of the relations $A\subseteq B$ or $B\subseteq A$ holds. A maximal chain is then the one maximal with respect to set inclusion. In our finite settings, a maximal chain $\mathcal A$ is of the form $\mathcal A = \{ A_0, A_1,  \ldots A_n \}$ where $A_0 = \emptyset, A_n = \states$ and $|A_{i+1}\backslash A_i| = 1$. 

The following proposition gives a previously known result  (see, e.g., \cite{bronevichAugustin2009, chateauneuf1989some, miranda2003extreme}). 
\begin{proposition}\label{prop-2-monotone-extremal}
	Let $L$ be a 2-monotone lower probability on a finite space $\states$ and $\mathcal A$ a maximal chain. Then there exists a linear prevision $P_\mathcal A\in \cset(L)$ such that $P(A) = L(A)$ for every $A\in \mathcal A$. 
	
	Moreover, every extreme point $P\in \cset(L)$ is of the form $P_\mathcal A$ for some maximal chain $\mathcal A$. 
\end{proposition}
The important claim of the above proposition is that the linear prevision with this property belongs to $\cset(L)$. 
This characterization of the extreme points of credal sets of 2-monotone lower probabilities can now be used to give a characterization of the normal cones of those credal sets. 
Two vectors $g, h\in \allgambles$ are said to be \emph{comonotone} if $g(x) > g(y)$ implies that $f(x)\ge f(y)$ for every pair $x, y\in \states$. A set of vectors $\mathcal H$ is said to be \emph{comonotone} if all pairs of its elements are. 
\begin{proposition}[\cite{den:97}]\label{prop-comonotonicity-equiv}
	Let $\mathcal H$ be a set of vectors. The following conditions are equivalent:
	\begin{enumerate}[(i)]
		\item $\mathcal H$ is comonotone; 
		\item the level set $\mathcal L = \{ \{ x\colon f(x) \le \alpha \} \colon f\in \mathcal H, \alpha\in \RR \}$ forms a chain;
		\item an ordering of elements $\states$ exists, denoted by $x_1, \ldots, x_n$, such that $f(x_i) \le f(x_{i+1})$ for every $f\in \mathcal H$ and $1\le i \le n-1$. 
	\end{enumerate}
\end{proposition}

\begin{lemma}\label{lemma-comonotone-cone}
	A set of vectors $\mathcal H$ is comonotone if and only if $\cone{\mathcal H}$ is comonotone. 
\end{lemma}
\begin{proof}
	The 'if' part being trivial, we prove the 'only if' part. Thus, assume $\mathcal H$ being comonotone. By Proposition~\ref{prop-comonotonicity-equiv}, the level set $\mathcal L$ then forms a chain. As by the assumption, $\states$ is finite, $\mathcal L$ must be a finite chain. Moreover, every $f\in \mathcal H$ is of the form $f = \sum_{L\in \mathcal L}\alpha_L \charf{L}$ for some non-negative coefficients $\alpha_L$.  Clearly, every positive linear combination of positive linear combinations of the indicator sets in $\mathcal L$ is again one; and its level sets belong to the same collection $\mathcal L$, which is therefore the collection of the level sets of all elements of the cone $\cone{\mathcal A}$. Now, since $\mathcal L$ is a chain, the cone must therefore be a comonotone set. 
\end{proof}

\begin{corollary}\label{cor-comonotone-chain}
	Let $\mathcal A$ be a collection of sets and $N_\mathcal A = \cone{\mathcal A}$. Then $N_\mathcal A$ is comonotone if and only if $\mathcal A$ forms a chain. 
\end{corollary}
\begin{proof}
	Clearly, $\mathcal A$ is a comonotone set exactly if it forms a chain, which by Lemma~\ref{lemma-comonotone-cone}, is exactly if the cone it generates is comonotone. 
\end{proof}
\begin{definition}
	A set of vectors $\mathcal H$ will be called \emph{maximal comonotone} if its level sets $\mathcal L$ form a maximal chain. 
\end{definition}
The following corollary is obvious. 
\begin{corollary}
	Let $\mathcal H$ be a maximal comonotone set. Then $\cone{\mathcal H}$ is a maximal comonotone set. 
\end{corollary}
\begin{theorem}\label{thm-2-monotone-cones}
	Let $\states$ be a finite set and let $\mathbb{A} = \{ \mathcal A\subset 2^\states \colon \mathcal A \text{ forms a chain} \}$. Denote with $\fan{\mathbb{A}} = \{ \cone{\mathcal A} \colon \mathcal A\in \mathbb{A} \}$ the collection of cones generated by the chains. The following propositions hold:
	\begin{enumerate}[(i)]
		\item $\fan{\mathbb{A}}$ is a complete fan.
		\item Let $L$ be a 2-monotone lower probability on $2^\states$ and $\mathcal A\subseteq 2^\states$ a maximal chain. Every $\cone{\setcol{A}}\in \fan{\mathbb{A}}$ is contained in a single normal cone $\ncone{L}{P}$.
		\item Every $\cone{\setcol{A}}\in \fan{\mathbb{A}}$ is a simplicial cone.
		\item $\fan{\mathbb{A}}$ is a complete normal simplicial fan. 
		\item Let $\mathcal A = \{ A_1, \ldots, A_n\}$ be a maximal chain. Then MESC $\cone{\mathcal A}$ is adjacent to exactly $n$ MESCs $\cone{\mathcal A_i}$ in $\mathbb{A}$ generated as follows. Take some $i\in \{ 1, \ldots, n\}$. Then form $\mathcal A_i$ by removing $A_i$ and replacing it with $A'_i = A_{i-1}\cup (A_{i+1}\backslash A_i)$. (We set $A_0 = \emptyset$.) 
		\item The graph with vertices $\mathbb{A}$ and edges corresponding to adjacency is $n$-regular and connected.
	\end{enumerate}
\end{theorem}
\begin{proof}
	The proof of (i) consists of the following three steps, whose proofs are carried out below.
	\begin{enumerate}[Step 1.]
		\item Every vector in $\allgambles$ is contained in at least one cone in $\fan{\mathbb{A}}$.
		\item The intersection of two cones in $\fan{\mathbb{A}}$ is a face of both.
		\item Faces of all cones are contained in $\fan{\mathbb{A}}$.
	\end{enumerate}
	Step 1: Take arbitrary vector $f\in \allgambles$ and let $\mathcal A$ be the collection of its level sets, which is always a chain. Then clearly $f$ is a positive linear combination of elements in $\mathcal A$ and therefore belongs to the cone $\cone{\setcol{A}}$. 
	
	Step 2: Let $\mathcal A_1$ and $\mathcal A_2$ be two chains in $\mathbb{A}$ and $\cone{\mathcal A_i}$ the corresponding cones. We will show that 
	\begin{equation}
		\cone{\mathcal A_1}\cap \cone{\mathcal A_2} = \cone{\mathcal A_1\cap \mathcal A_2}.
	\end{equation}
	The inclusion $\supseteq$ is clear, whence it remains to show the $\subseteq$ part. Thus let $f\in \cone{\mathcal A_1}\cap \cone{\mathcal A_2}$. Take only those members of $\mathcal A_1$ that have strictly positive coefficients in $\sum_{A\in \mathcal A_1} \alpha_A \charf{A}$. It is easy to check that the level sets of $\mathcal F$ are exactly these sets. And by the same argument, the level sets must be exactly those with strictly positive coefficients in the positive linear combinations of sets from $\mathcal A_2$. Hence, these sets must be the same, and therefore lie in the intersection  $\mathcal A_1\cap \mathcal A_2$. It also follows by the construction of this argument that these sets form a chain, and thus generate a subcone of a possibly lower dimension that is a face both cones.  
	
	Step 3: A face of $\cone{\mathcal A}$ is $\cone{\mathcal A'}$ generated by a subset $\mathcal A'$ of $\mathcal A$. Clearly a subset of a chain is again a chain, and therefore $\cone{\mathcal A'}$ belongs to  $\fan{\mathbb{A}}$ by definition. 
	
	(ii) Let $\cone{\mathcal A}$ be given corresponding to a maximal chain $\mathcal A$. It follows from Proposition~\ref{prop-2-monotone-extremal} that $P_\mathcal A$ is an extreme point of $\cset(L)$ such that $P_\mathcal A(A) = L(A)$ for every $A\in \mathcal A$. Hence, $\mathcal A$ belongs to $\ncone{L}{P_\mathcal A}$ and so does the cone its generates. 
	
	(iii) Let $\mathcal A$ be a chain. The vectors $\charf{A}$ for $A\in \mathcal A$ are then clearly linearly independent. To see this, combine the corresponding 0-1 vectors to a matrix which has clearly a triangular form. Hence $\cone{\mathcal A}$ is simplicial by definition. 
	
	(iv) follows directly from the above. 
	
	(v) It follows clearly from the construction that collections $\mathcal A_i$ are again chains and therefore $\cone{\mathcal A_i}\in \mathbb{A}$ for every $i=1, \ldots, n$. It is also easy to see that sets $A'_i$ are the only candidates to replace $A_i$ so that the resulting collection is still a chain. 
	
	(vi) It follows immediately from (v) that every $\cone{\mathcal A}$ has exactly $n$ neighbours. 
\end{proof}
The following result previously known from literature \cite{chateauneuf1989some, shapley1971cores, WALLNER2007339} is as an easy corollary.
\begin{corollary}
	The credal set $\cset(L)$ of a 2-monotone lower probability $L$ in the probability space $(\states, 2^{\states})$ has at most $n!$ extreme points, where $n=|\states|$.  	
\end{corollary} 
\begin{proof}
	By Theorem~\ref{thm-2-monotone-cones}(ii), the extreme points are in a one-to-one correspondence with maximal chains in $\mathbb{A}$. Moreover, the maximal chains are in a one-to-one correspondence with permutations. To see this, notice that each permutation $\sigma$ induces the natural chain $\mathcal A_\sigma = \{ A_i \colon A_i = \{ x_{\sigma(1)}, \ldots, x_{\sigma(i)} \}, i = 1, \ldots, n \}$ and that the mapping $\Sigma\to 2^{\states}$ that maps $\sigma\mapsto \mathcal A_\sigma$ is bijective. The number of maximal chains is therefore equal to $n!$ and since every cone generated by a maximal chain is contained in an $n$-dimensional normal cone. This limits the number of normal cones to at most $n!$.  
\end{proof}

A simple proof for the representation of comonotonic additive functionals with 2-monotone lower probabilities follows. This result is known in several forms (\cite{den:97, sch:86}), yet we present it here as an application of our results presented above. 
\begin{corollary}\label{cor-com-ad-2-mon}
	Let $\low E$ be a lower expectation functional. The following conditions are equivalent:
	\begin{enumerate}[(i)]
		\item $\low E$ is comonotonic additive;
		\item a 2-monotone lower probability $L$ exists so that $\low E$ is its natural extension.
	\end{enumerate}
\end{corollary}
\begin{proof}
	We start with proving (i) $\implies$ (ii). Let $L(A) = \low E(\charf{A})$ for every subset of $\states$. Take sets $A$ and $B$. Then $\charf{A\cap B}$ and $\charf{A\cup B}$ are comonotone and $\charf{A\cap B} + \charf{A\cup B} = \charf{A} + \charf{B}$. Moreover, superadditivity of coherent lower previsions implies that $\low E(\charf{A}) + \low E(\charf{B}) \le \low E(\charf{A} + \charf{B}) = \low E(\charf{A\cap B}) + \low E(\charf{A\cup B})$, which amounts to 2-monotonicity of $L$. 
	
	It remains to show that $\low E$ is the natural extension of $L$. To see this, take an arbitrary vector $f\in \allgambles$. Then $f$ can be represented as a positive linear combination of its level sets, which also form a chain, say $\mathcal A$. Now all vectors $\charf{A}$ for $A\in \mathcal A$ and $f$ form a comonotone set and $f=\sum_{A\in \mathcal A} \alpha_A \charf{A}$. By comonotonic additivity we have that 
	\begin{equation*}
			\low E(f) = \sum_{A\in \mathcal A} \alpha_A \low E(A) = \sum_{A\in \mathcal A} \alpha_A L(A) = \low E_L(f),
	\end{equation*} 
	where $\low E_L$ denotes the natural extension of $L$. Hence, both lower previsions coincide. 
		
	(ii) $\implies$ (i) is a well known property of 2-monotone lower probabilities. It is directly implied by Theorem~\ref{thm-2-monotone-cones}(ii) and Proposition~\ref{prop-cone-additivity}.  
\end{proof}

\section{Normal cones of probability intervals models}\label{s-ncpri}
According to the representation of credal sets \eqref{eq-cset-constraints-first}--\eqref{eq-constraint-equality}, the support functionals of credal sets induced by probability intervals are in the forms of $\charf{\{x \}}$ and $\charf{\{ x\}^c}$ respectively. In this section we propose a full characterization of the complete normal simplicial fans for coherent PRI models. In contrast with the case of 2-monotone lower probabilities, their structure is not unique, or in other words, they are not normally equivalent.  
\subsection{Maximal elementary simplicial cones corresponding to coherent PRI models}
Next we provide a general form of MESCs corresponding to coherent PRI models. That is, every MESC of a coherent PRI model is a cone of the described form, yet, not every cone of this form is a MESC of a particular coherent PRI model. Recall that in the case of 2-monotone lower probability, every maximally comonotone cone is a MESC of every 2-monotone lower probability model, while in the case of coherent PRI models, the actual polyhedral structure needs to be further determined. This analysis is postponed to Section~\ref{ss-rncep}. 

The following Proposition~\ref{prop-PRI-mesc} provides the description of the possible MESCs in terms of collections of sets $\mathcal A$ used to form the cone using construction \eqref{eq-cone-construction}.  
\begin{proposition}\label{prop-PRI-mesc}
	Let $\cone{\setcol{A}}$ be a maximal elementary simplicial cone of a credal set corresponding to a coherent PRI model $(l, u)$. Then $\setcol{A}$ is of the following form.
	Enumerate the elements of $\states$ in some order and choose $k$ so that  $1\le k \le n-2$. Now the set $\setcol{A}$ consists of the following $n$ sets
	\begin{enumerate}[(i)]
		\item singleton sets $A_i = \{x_i\}$ for $1\le i \le k$;
		\item complements of the singletons $A_j = \{ x_j \}^c$ for $k+1\le j \le n-1$;
		\item set $\states$.  
	\end{enumerate}
\end{proposition}
\begin{proof}
	By the construction of a coherent PRI model, the maximal simplicial cones are generated by sets of the described form and obviously, $x_i\neq x_{i'}$ for $i\neq i'; i, i' \le k$ and $x_j\neq x_{j'}$ for $j\neq j'; j, j' > k$. Thus, it remains to show that:
	\begin{enumerate}
		\item $x_i\neq x_j$ for every $i\le k$ and $j>k$. Indeed, if $x_i = x_j$ for some pair of indices, then $\{ x_i \} \cup \{x_i\}^c = \states$ and therefore the set loses linear independence. 
		\item $k\ge 1$. Suppose contrary that $k=0$. Then $\mathcal A = \{ \{ x_i \}^c \colon x_i\in \{ x_n \}^c \}\cup \{\states \}$. Then, $\sum_{A\in \mathcal A} \charf{A} = n\charf{\states} + \charf{\{x_n\}}$, whence $\charf{\{x_n\}} = \sum_{A\in \mathcal A} \charf{A} - n\charf{\states}$ and thus $\charf{\{x_n\}}$ belongs to the cone generated by $\mathcal A$. (Notice that $\charf{\states}$ may appear with a negative coefficient by \eqref{eq-nc-form}). This contradicts Proposition~\ref{prop-full-triangulation}(ii). 
		\item $k<n-1$. This case is symmetrical to case 2. By assuming $k=n-1$ we would then have $\mathcal A = \{ \{x_i\}\colon 1\le i \le n-1 \} \cup \{\states \}$, whence $\sum_{A\in \mathcal A} \charf{A} = \charf{\{ x_n \}^c}$ which again violates Proposition~\ref{prop-full-triangulation}(ii).
	\end{enumerate}
\end{proof}
Using the above proposition, we can now easily deduce the general form of elements of the cones. Notice that $x_n$ in the above proposition, as well as in the next corollary, is the only element of $\states$ such that neither $\{x_n\}$ nor $\{ x_n\}^c$ belongs to $\mathcal A$. 
\begin{corollary}\label{cor-cones-PRI}
	Let $\cone{\mathcal A}$ be a cone corresponding to a set $\mathcal A$ constructed as in Proposition~\ref{prop-PRI-mesc}. Then $h\in \cone{\mathcal A}$ if and only if 
	\begin{enumerate}[(i)]
		\item $h(x_i) \ge h(x_n)$ for $1\le i \le k$ and
		\item $h(x_j) \le h(x_n)$ for $k< j \le n-1$.
	\end{enumerate}	  	
\end{corollary}
\begin{proof}
	First we notice that every vector from $\mathcal A$ satisfies conditions (i) and (ii), and so does its every positive multiple. It is also clear that the sum of any two vectors complying with conditions (i) and (ii), satisfies them too. The set of functions satisfying conditions (i) and (ii) is therefore a cone that contains $\cone{\mathcal A}$. 
	
	To verify that every $h$ satisfying (i) and (ii) is actually in $\cone{\mathcal A}$, we first notice that vectors $\charf{\{ x_j\}^c}$ can be replaced by $-\charf{ \{x_j\}}$, as $-\charf{ \{x_j\}} = \charf{\{ x_j\}^c} - \charf{\states}$. Now we take $\alpha_n = h(x_n), \alpha_i = h(x_i)-h(x_n)$ for $1\le i\le k$ and $\alpha_j = h(x_n)-h(x_j)$, which gives 
	\begin{align}
		h & = \sum_{i=1}^k \alpha_i \charf{\{x_i\}} + \sum_{j=k+1}^{n-1} \alpha_j (-\charf{\{x_j\}}) + \alpha_n \charf{\states} \\ 
		& = \sum_{i=1}^k \alpha_i \charf{ \{x_i\}} + \sum_{j=k+1}^{n-1} \alpha_j \charf{ \{x_j\}^c} + \left(\alpha_n - \sum_{j=k+1}^{n-1} \alpha_j\right)\charf{\states},
	\end{align} 
	where all $\alpha_i\ge 0$ for $1\le i \le n-1$. 
\end{proof}
Proposition~\ref{prop-PRI-mesc} could be compared with Theorem~\ref{thm-2-monotone-cones}(ii), and the relation between vectors described by Corollary~\ref{cor-cones-PRI} with comonotonicity. Applying Proposition~\ref{prop-cone-additivity} to this case would imply that coherent lower previsions induced by PRI models are additive within sets of the form described by Corollary~\ref{cor-cones-PRI}. Moreover, we will see in the sequel that this form of additivity also implies comonotone additivity, consequently implying the previously known fact that coherent PRI models are 2-monotone. 

In the following text $f|_A \le \alpha$ means $f(x)\le \alpha$ for every $x\in A$, and $f|_A < \alpha$ denotes the strict inequality $f(x)< \alpha$ for every $x\in A$. We adopt analogous notation for other relations ($>, \ge, =$). 
\begin{definition}\label{defn-cones-PRI}
	Let $x\in \states$ and $A, B\subset \states$ be given such that $A\neq \emptyset, B\neq \emptyset, x\not\in A\cup B$ and $A\cap B=\emptyset$. Then we define the following set of vectors:
	\begin{equation}\label{eq-cones-PRI}
		N(x, A, B) = \{ f\in \allgambles\colon f|_A\le f(x), f|_B\ge f(x), f|_{\states\backslash (A\cup B)} = f(x) \}.
	\end{equation}
\end{definition}
\begin{corollary}\label{cor-ncone-pri-2}
	Let $N(x, A, B)$ satisfy the assumptions of Definition~\ref{defn-cones-PRI}. Then $N(x, A, B) = \cone{\mathcal A}$ where $\mathcal A = \{ \{x\}\colon x\in B\}\cup \{ \{y\}^c\colon y\in A\}$. 
\end{corollary}
\begin{proof}
	The proof is almost identical to the proof of Corollary~\ref{cor-cones-PRI}.
\end{proof}
\begin{theorem}\label{thm-PRI-cones}
	Let $x\in \states$ and $A, B\subseteq \states$ be given such that $A\neq \emptyset, B\neq \emptyset$ and $A\cap B=\emptyset$. Then the following conditions hold:
	\begin{enumerate}[(i)]
		\item $N(x, A, B)$ is a cone.
		\item A coherent PRI model $(l, u)$ exists such that $N(x, A, B)$ is a normal cone of its credal set. 
		\item $A\cup B = \{ x \}^c$ if and only if $N(x, A, B)$ is MESC. 
		\item If $A'\subseteq A, B'\subseteq B$ then $N(x, A', B')$ is a face of $N(x, A, B)$. 
		\item $\ri{N(x, A, B)} = \{ f\in \allgambles\colon f|_A< f(x),  f|_B > f(x) , f|_{(A\cup B)^c} = f(x) \}$.
		\item For every vector $f\in\allgambles$ and every $x\in\states$, $f\in \ri{N(x, [x]^f_-, [x]^f_+)}$, where $[x]^f_+ = \{ y\colon f(y) > f(x) \}$ and $[x]^f_- = \{ y\colon f(y) < f(x) \}$, provided that neither $[x]^f_+$ nor $[x]^f_-$ is empty. Moreover, $f\in \ri{N(x, A, B)}$ if and only if $A=[x]^f_-$ and $B=[x]^f_+$. 
		\item Every vector $f\in \allgambles$ belongs to at most $n-2$ cones of the form \eqref{eq-cones-PRI}.
		\item Let $N(x, A, B)$ and a cone of the form $\cone{\mathcal A}$, where $\mathcal A$ is a maximal chain, be given. Then either $\ri{\cone{\mathcal A}}\subseteq N(x, A, B)$ or $\ri{\cone{\mathcal A}}\cap N(x, A, B)=\emptyset$. 
	\end{enumerate}
\end{theorem}
\begin{proof}
	(i) is an immediate consequence of definitions and Corollary~\ref{cor-cones-PRI}. 
	
	(ii): We construct a coherent PRI model $(l, u)$ such that $N(x, A, B)$ is its normal cone. Take any linear prevision $P$ with the corresponding probability mass function $p$ such that for some $\varepsilon>0, \min\{ p(x)\colon x\in \states  \} > \varepsilon$. Now set $l(x) = p(x)$ for every $x\in B$, $u(x') = p(x')$ for every $x'\in A$. Further set $l(y) = p(y)-\varepsilon$ for every $y\in B^c$ and $u(y') = p(y')+\varepsilon$ for every $y'\in A^c$. By construction, all $l(x)$ and $u(x)$ lie within $[0, 1]$ interval . 
	
	For coherence, all bounds must be reachable by elements of the credal set. Thus take some $y\in B^c$ and set $p'(y) = p(y)-\varepsilon = l(y)$, for some $x\in A^c$ such that $x\neq y$, set $p'(x) = p(x) + \varepsilon = u(x)$ , and $p'(z) = p(z)$ otherwise. Since $A$ and $B$ are non-empty and disjoint, it is clear that pairs $x, y$ can be chosen so that every lower and every upper bound are reached by some $p'$. 
	
	Let $\low E$ denote the natural extension of $(l, u)$. By construction, it is clear that $P(\charf{ \{x\} }) = \low E(\charf{ \{x\} })$ for every $x\in B$ and $P(\charf{\{ y\}^c}) = 1-P(\charf{\{ y\}}) = 1- \up E(\charf{\{ y\}}) = \low E(\charf{\{ y\}^c})$ for every $y\in A$, which by Proposition~\ref{prop-cone-positive-hull} implies that $\ncone{\cset(l, u)}{P} = \cone{\mathcal A}$, where $\mathcal A = \{ \{x\}\colon x\in B\}\cup \{ \{y\}^c\colon y\in A\}$, which by Corollary~\ref{cor-ncone-pri-2} further implies that  $\ncone{\cset(l, u)}{P} = N(x, A, B)$. 
	
	(iii) follows directly from Proposition~\ref{prop-PRI-mesc}.
	
	(iv) the set $N(x, A', B')$ is clearly a subset of $N(x, A, B)$ where some inequality constraints are replaced by equalities. Hence, $N(x, A', B')$ is a face of $N(x, A, B)$.
	
	(v) Suppose $f\in N(x, A, B)$ and $f(y)=f(x)$ for some $y\in A$. Then $f\in N(x, A\backslash\{y\}, B)$ which by (iv) is a face of $N(x, A, B)$.
	
	The relation  $f\in \ri{N(x, [x]^f_-, [x]^f_+)}$ in (vi) is a direct consequence of the construction and Corollary~\ref{cor-cones-PRI}. It is also clear by the definition that $f\in N(x, A, B)$ implies that $[x]^f_-\subseteq A$ and $[x]^f_+\subseteq B$; however, if any of these set inclusions is strict, then $N(x, [x]^f_-, [x]^f_+)$ is a proper face of $N(x, A, B)$, by (iv). But an element cannot be contained in relative interiors of a polyhedron and its proper face at the same time. Thus, the set inclusion must in fact be equality relation. 
	
	We have that $[x]^f_+=\emptyset$ if $x=\arg\max_{x\in \states} f(x)$ and $[x]^f_-=\emptyset$ if $x=\arg\min_{x\in \states} f(x)$. For those choices of $x$, the cones of the form \eqref{eq-cones-PRI} therefore do not exist. So, cones of this form only exist for at most $n-2$ choices of $x$, and for every such choice only one such cone exists, by (vi), which consequently confirms (vii). 
	
	The set $\ri{\cone{\mathcal A}}$ contains comonotone vectors with strictly different components. Let some $x\in \states$ be given. It is easy to see that any two such functions, say $f$ and $g$ satisfy $[x]^f_+ = [x]^g_+$ and $[x]^f_- = [x]^g_-$ and that $[x]^f_+\cup [x]^f_- = \{x\}^c$, which makes $N(x, [x]^f_-, [x]^f_+)$ the only cone of the form $N(x, A, B)$ that contains $f$ and $g$, and with them the entire relative interior. Thus (viii) is proved. 
\end{proof}
\begin{corollary}
	Let $(l, u)$ be a coherent PRI model. Then:
	
	\begin{enumerate}[(i)]
		\item Its natural extension $\low E$ is comonotonic additive.   
	\item A 2-monotone lower probability $L$ exists that coincides with $(l, u)$ in the sense that $L(\{ x\}) = l(x)$ and $L(\{ x\}^c) = 1-u(x)$ for every $x\in \states$ and whose natural extension coincides with $\low E$. 
	\end{enumerate} 
\end{corollary}
\begin{proof}
	(i) is a direct consequence of Proposition~\ref{prop-cone-additivity} and Theorem~\ref{thm-PRI-cones}(viii) and (ii) follows directly from Corollary~\ref{cor-com-ad-2-mon} and (i). 
\end{proof}
Property (ii) in the above corollary is already known in literature, see e.g. \cite{decampos:94, weic:01}, and (i) is then its easy consequence. The reason we include it here is that its proof as presented here is a direct application of our approach, and especially illustrating the importance if property (viii) of Theorem~\ref{thm-PRI-cones}, which could be understood as an alternative characterization of comonotone additivity.  

\subsection{Relating normal cones to extreme points}\label{ss-rncep}
Normal cone structure described in the above proposition is closely related to the extreme points of the credal sets corresponding the probability intervals. Their characterization is known from the literature (see e.g. \cite{decampos:94}). Let $(l, u)$ be a probability interval model on $\states$, $\cset$ its credal set and $h\in \allgambles$ a vector. Let $P$ be a the extremal linear prevision such that $P(h) = \min_{P\in\cset} P(h)$. To construct $P$, let $x_1, \ldots, x_n$ be an enumeration of the elements of $\states$ such that $h(x_i)\le h(x_{i+1})$. Let $k$ be an index such that 
\begin{equation}\label{eq-pri-index-k}
	l(x_k) \le 1-\sum_{i=k+1}^n l(x_i)-\sum_{i=1}^{k-1} u(x_i) \le u(x_k). 
\end{equation}
Then take 
\begin{equation}
	P(x_i) = \begin{cases}
		u(x_i), & i < k; \\
		l(x_i), & i > k; \\
		1-\sum_{i=1}^k u(x_i) - \sum_{i=k+1}^n l(x_i), & i = k. 
	\end{cases}
\end{equation}
The proof that so defined $P$ minimizes $P(h)$ over $\cset$ can be found in \cite{decampos:94}. Denote $A = \{ x_{1}, \ldots, x_{k-1} \}$ and $B = \{ x_{k+1}, \ldots, x_{n}\}$. It follows directly from the construction that given another vector $h'$, the minimizing $P$ is the same whenever the induced sets $A$ and $B$ are the same for $h$ and $h'$. In our terms of normal cones, such vectors $h$ and $h'$ both lie in the same normal cone $N(x_k, A, B)$. 
\begin{remark}
	The case where $k$ satisfying equation \eqref{eq-pri-index-k} equals 1 or $n$, deserves an additional illumination. These two cases correspond to the cones of the form $N(x, \emptyset, B)$ and $N(x, A, \emptyset)$ respectively, which have been shown not to be maximal elementary simplicial cones. The cases can be treated in a symmetric way; therefore, we only consider the case $k=1$. Equation \eqref{eq-pri-index-k} then gives that $l(x_1) \le 1-\sum_{i=2}^n l(x_i) \le u(x_1)$, which for coherent PRI model can only be satisfied if $1-\sum_{i=2}^n l(x_i) = u(x_1)$. It follows that $l(x_2) = 1-u(x_1)-\sum_{i=3}^n l(x_i)$. Hence, $k=2$ also satisfies \eqref{eq-pri-index-k}. The cone corresponding to this case is $N(x_2, \{x_3, \ldots, x_n\}, \{ x_1\})$, which is clearly a proper subcone of $N(x_1,  \{x_2, \ldots, x_n\}, \emptyset)$, thus confirming that the latter is not elementary. 
	In general, the cone of the form $N(x, \emptyset, B)$ is a simplicial complex of cones $N(y, \{ x \}, B\backslash \{y\})$ with non-intersecting interiors.  
\end{remark}
\begin{example}\label{ex-pri-mesc}
	Let $\states = \{ x_1, x_2, x_3, x_4\}$ and let $(l, u)$ be a PRI model given by the vector of lower bounds $l = (\frac{1}{5}, \frac{1}{5}, \frac{1}{5}, \frac{1}{5})$ and upper bounds $u = (\frac{1}{3}, \frac{1}{3}, \frac{1}{3}, \frac{1}{3})$, and denote its natural extension by $\low E$. Let $h = (4, 1, 3, 2)$. By the above construction, it is easy to see that $\low E(h) = P(h)$, where $P$ is the linear prevision corresponding to probability mass function $p = (\frac{1}{5}, \frac{1}{3}, \frac{1}{5}, \frac{4}{15})$, resulting in $\low E(h) = \frac{34}{15}$. The corresponding MESC is $N(x_4, \{ x_2\}, \{ x_1, x_3\})$ containing all vectors $h'$ with $h'(x_2) \le h'(x_4), h'(x_1)\ge h'(x_4)$ and $h'(x_3)\ge h'(x_4)$. Take, for instance, $h' = (1, -1, 2, 0)$. Additivity within normal cone, Proposition~\ref{prop-cone-additivity}, implies that $\low E(h+h') = \low E(h) + \low E(h') = \frac{34}{15}+\frac{4}{15} = \frac{38}{15}$. Notice that $h$ and $h'$ are not comonotone, hence, additivity with respect to the normal cone applies to a larger class of vectors than comonotone additivity. 
\end{example}

\subsection{Graph structure of the normal cones corresponding to coherent PRI models}
We now analyze the adjacency relations for the family of cones of the form $N(x, A, B)$. The cone of this form is generated by the indicator functions of the family of sets 
\begin{equation}
	\mathcal A = \{ \{z \} \colon z\in B \} \cup \{ \{ v\}^c \colon v\in A \} \cup \{ \states \}.
\end{equation}
The adjacent cones are formed by selecting an element of $\mathcal A$ and replacing it by a suitable set to form a new cone. Not all candidates produce adjacent cones, though. To select those that do, we will make use of Lemma~\ref{lem-adjacent-cone}. Assume for the moment that $|A|, |B|>1$. We will return to the borderline cases later. Take some $y\in A$ and consider possible candidates for the replacement of $\{ y\}^c\in \mathcal A$. These are $\{ y\}, \{ x\}$ and $\{ x\}^c$. To see which induce adjacent cones, calculate the normal vector $t$ to the hyperplane $\mathrm{lin}(\mathcal A\backslash \{ \{y\}^c\})$ and denote its elements by $t(x)$ for $x\in\states$. For every $\{ x\}$, $\charf{\{x \}}\cdot t = 0$ implies that $t(x) = 0$. Similarly, $t\cdot \charf{\{ x\}^c} = t\cdot (\charf{\states}-\charf{\{x \}}) = -t\cdot \charf{\{x \}}$, because of $\states \in \mathcal A\backslash \{ \{y\}^c\}$, implies $t(x) = 0$ as well. Thus, because of $t\cdot \charf{\states} =0$, we have that $t(y) + t(x) = 0$. Take the solution where $t(x) = 1$ and $t(y) = -1$.  Since $\charf{\{ y\}^c}\cdot t = 1$, the scalar product of the new vector with $t$ must be negative. The products of the candidates identified above are the following: $\charf{\{y\}}\cdot t = -1, \charf{\{x \}}\cdot t = 1, \charf{\{ x\}^c}\cdot t = -1$. Thus the candidates that induce adjacent cones are $\{ y \}$ and $\{ x\}^c$, which gives us the following adjacent cones 
\begin{enumerate}[({A}1)]
	\item $N(x, A\backslash \{ y\}, B\cup \{ y\})$,
	\item $N(y, (A\backslash \{ y\})\cup \{ x\}, B)$.
\end{enumerate}
Let us now consider the case of $z\in B$, where $\{ z\}\in \mathcal A$. The candidates to replace $\{ z\}$ are again $\{ z\}^c, \{x\}$ and $\{ x\}^c$. The same analysis as above now gives us a normal vector $t$, such that $t(x) = 1$ and $t(z) = -1$. Now we have that $t\cdot \charf{\{z\}} = -1$, whence the candidates with positive scalar product are those inducing adjacent cones. We have $\charf{\{ z\}^c}\cdot t = 1, \charf{\{x \}}\cdot t = 1, \charf{\{ x\}^c}\cdot t = -1$. Now $\{ z\}^c$ and $\{ x\}$ fit, inducing the cones:
\begin{enumerate}[({B}1)]
	\item $N(x, A\cup \{ z\}, B\backslash \{ z\})$,
	\item $N(z, A, (B\backslash \{ z\})\cup \{ x\})$.
\end{enumerate}
Now, in the case where $|A| = 1$, only (A2) is possible in the first case, while (B2) is the only possible neighbour if $|B|=1$. 

In particular case of a coherent PRI model $(l, u)$, only one of adjacent cones (A1) or (A2) and (B1) or (B2) respectively corresponds to an extreme point. Let us again first consider the case of $y\in A$. Whether the adjacent cone is (A1) or (A2) depends on which $x$ or $y$ satisfies condition \eqref{eq-pri-index-k} in place of $x_k$. The fact that $N(x, A, B)$ corresponds to an extreme point, implies that 
\begin{equation}\label{eq-cone-starting}
	l(x) \le 1-\sum_{v\in A} l(v)-\sum_{z\in B} u(z) \le u(x).
\end{equation}
Because of $y\in A$, it easily follows that 
\begin{equation}
	l(y) \le 1-\sum_{v\in (A\backslash\{y\})\cup \{ x\}} l(v)-\sum_{z\in B} u(z), 
\end{equation}
which corresponds to replacing $y$ with $x$ in $A$, as in the case of (A2). On the other hand, we can either have 
\begin{equation}\label{eq-cone-A1}
	l(x) \le 1-\sum_{v\in A\backslash\{y\}} l(v)-\sum_{z\in B\cup\{y\}} u(z),
\end{equation}
which is equivalent to 
\begin{equation}
	1-\sum_{v\in (A\backslash\{y\})\cup \{ x\}} l(v)-\sum_{z\in B} u(z) \ge u(y). 
\end{equation}
or the opposite inequalities in both equations. If the first inequality holds, then (A1) is the cone corresponding to an extreme point, because $x$ is the element satisfying  \eqref{eq-pri-index-k}. In the opposite case, where 
\begin{equation}\label{eq-cone-A2}
	1-\sum_{v\in (A\backslash\{y\})\cup \{ x\}} l(v)-\sum_{z\in B} u(z) \le u(y),
\end{equation}
condition \eqref{eq-pri-index-k} is satisfied by $y$ and therefore (A2) corresponds to an extreme point. 

Now take some $z\in B$. Again, equation \eqref{eq-cone-starting} implies 
\begin{equation}
	1-\sum_{y\in A} l(y)-\sum_{v\in (B\backslash\{z\})\cup \{ x\}} u(z) \le u(z),  
\end{equation}
corresponding to replacing $z$ with $x$ in $B$. Furthermore, we have the following pair of equivalent equations
\begin{equation}\label{eq-cone-B1}
	1-\sum_{y\in A\cup\{z\}} l(y)-\sum_{v\in (B\backslash\{z\})} u(z) \le u(x)  
\end{equation}
and 
\begin{equation}
	l(z) \ge 1-\sum_{y\in A} l(y)-\sum_{v\in (B\backslash\{z\})\cup \{ x\}} u(z) .  
\end{equation}
If \eqref{eq-cone-B1} holds, then (B1) corresponds to an extreme point and in the case of the opposite inequality 
\begin{equation}\label{eq-cone-B2}
	l(z) \le 1-\sum_{y\in A} l(y)-\sum_{v\in (B\backslash\{z\})\cup \{ x\}} u(z) , 
\end{equation}
(B2) corresponds to an extreme point. 

Let us now summarize. 
\begin{theorem}\label{thm-pri-adjacency}
	Let a cone of the form $N(x, A, B)$ correspond to an extreme point of a coherent PRI model $(l, u)$. 
	\begin{enumerate}[(i)]
		\item The adjacent cones are then exactly the cones of the form 
		\begin{itemize}
			\item (A2) for every $y\in A$; 
			\item if $|A|>1$, (A1) for every $y\in A$;
			\item (B2) for every $z\in B$;
			\item if $|B|>1$,  (B1) for every $z\in B$.
		\end{itemize} 
		\item For a given $y\in A$, (A1) corresponds to an extreme point if \eqref{eq-cone-A1} is satisfied; and (A2) corresponds to an extreme point if \eqref{eq-cone-A2} is satisfied. 
		\item For a given $z\in B$, (B1) corresponds to an extreme point if \eqref{eq-cone-B1} is satisfied; and (B2) corresponds to an extreme point if \eqref{eq-cone-B2} is satisfied. 
	\end{enumerate}	
\end{theorem}
\begin{remark}
	The borderline case in the above theorem is if both \eqref{eq-cone-A1} and \eqref{eq-cone-A2} or \eqref{eq-cone-B1} and \eqref{eq-cone-B2} are satisfied. This case corresponds to the situation where the two cones correspond to different possible triangulations of the same normal cone. 
\end{remark}
We illustrate the above theorem with the following example. 
\begin{example}\label{ex-pri-adjacency}
	Consider again the PRI model from Example~\ref{ex-pri-mesc}. It was identified that $N(x_4, \{ x_2\}, \{ x_1, x_3\})$ is a MESC in the normal simplicial fan of the model. Let us now consider all adjacent MESCs. The credal set $\cset(l, u)$ is three dimensional, whence every MESC is adjacent to three other MESCs. 	According to notation in Theorem~\ref{thm-pri-adjacency}, we have $x = x_4, A = \{ x_2\}$ and $B = \{ x_1, x_3\}$. The adjacent cones are obtained by modifying sets $A$ and $B$ as described in the theorem. Let us first consider the only element $y=x_2$ of $A$. By (i) of the theorem, the only possible adjacent cone is obtained using (A2) (because of $|A|=1$), which gives $N(x_2, \{ x_4\}, \{ x_1, x_3\})$. The other two adjacent MESCs are obtained by considering elements $z$ of $B$. Take first $z = x_1$, which gives us two possible adjacent cones, using (B2), would result in $N(x_1, \{x_2\}, \{x_3, x_4\})$, and using (B1), in $N(x_4, \{ x_1, x_2\}, \{ x_3\})$. In our particular case, only one of the two cones actually corresponds to an extreme point, and to select the right one we use the criterion (iii) from the theorem. An easy calculation shows that \eqref{eq-cone-B2} is satisfied, while \eqref{eq-cone-B1} is not, whence we conclude that the adjacent cone is  $N(x_1, \{x_2\}, \{x_3, x_4\})$. In a similar way we find the third adjacent cone to be $N(x_3, \{x_2\}, \{x_1, x_4\})$. 
	
	By symmetry, we can easily describe the complete structure of the normal simplicial fan. That is, it contains all 12 possible MESCs of the form $N(x, A, B)$, where $|A|=1$ and $|B|=2$. Moreover, every MESC of the form $N(x, \{y\}, \{z, t\})$ is adjacent to $N(y, \{x\}, \{z, t\}), N(z, \{y\}, \{x, t\})$ and $N(t, \{y\}, \{z, x\})$. 
\end{example}

\subsection{Maximal number of extreme points}
In this section we estimate the possible maximal number of extreme points of credal sets of coherent PRI models. By Proposition~\ref{prop-full-triangulation} (iv), every MESC is a normal cone of some convex set corresponding to an extreme point. In general however, a normal cone in an extreme point can be triangulated as a union of MESCs. Thus, the maximal number of extreme points of a credal set is bounded by the number of MESCs corresponding to a complete simplicial fan obtained as a triangulation of the normal fan of the credal set. Moreover, by Theorem~\ref{thm-PRI-cones}(viii), every cone generated by a maximal chain is contained in a single MESC corresponding to credal set of a coherent PRI model. As the number of maximal chains is known to be equal to $n!$, the number of MESCs and therefore the maximal number of extreme points can be estimated from the number of the chain generated cones contained in the MESCs. We start with the following simple result. 
\begin{proposition}\label{prop-number-cones-PRI}
	A cone of the form $N(x, A, B)$ contains exactly $|A|!\cdot |B|!$ cones of the form $\cone{\mathcal A}$, where $\mathcal A$ is a maximal chain. 
\end{proposition}
\begin{proof}
	By Corollary~\ref{cor-comonotone-chain}, there is a one-to-one correspondence between comonotone classes and chains, and thus also between maximal comonotone classes and maximal chains. Further, a maximal comonotone classes correspond to strict linear orderings in $\states$. A cone of the form $N(x, A, B)$ contains all functions $f$ satisfying $f|_A \le f(x) \le f|_B$. This induces a partial ordering $A\preceq x \preceq B$, which is compatible with exactly $|A|!$ complete orderings of $A$ and $|B|$ complete orderings in $B$, which gives exactly  $|A|!\cdot |B|!$ distinct complete orderings. 
\end{proof}
\begin{corollary}\label{cor-number-cones-bounds}
	The number $m$ of distinct MESCs in the triangulation of a normal cone of a credal set corresponding to coherent PRI models on a set $\states$ with $n$ elements satisfies the following inequality 
	\begin{equation}\label{eq-number-cones-bounds}
		\frac{n!}{(n-2)!} = n(n-1) \le m \le  \frac{n!}{\left \lfloor\frac{n-1}{2} \right \rfloor !\cdot \left \lceil\frac{n-1}{2} \right \rceil !}. 
	\end{equation}
\end{corollary}
\begin{remark}
	The last inequality in \eqref{eq-number-cones-bounds} is known from the literature, and can be found in \cite{decampos:94}, where the estimate for the number of extreme points of coherent PRI on 10 points is reported as an example, and it coincides with the number obtained in our Example~\ref{ex-number-ep1}. 
\end{remark}
\begin{proof}
	By Proposition~\ref{prop-number-cones-PRI}, the number of maximal comonotone cones contained within a MESC is between $(n-2)!$ and $\left \lfloor\frac{n-1}{2} \right \rfloor ! \cdot \left \lceil \frac{n-1}{2} \right \rceil !$, which readily implies the proposed bounds. 
\end{proof}
The following examples demonstrate that both bounds are reachable. 
\begin{example}\label{ex-number-ep1}
	Let $n = |\states| = 10$ and set $l(x) = \frac{1}{11}$  and $u(x) = \frac{1}{9}$ for every $x\in \states$. Let $x_1, \ldots, x_{10}$ be an ordering of the elements. Taking $k=5$ we have that $1-\sum_{6}^{10}\frac 1{11} - \sum_{1}^{4}\frac 19 = 1-\frac{5}{11}-\frac{4}{9} = \frac{10}{99}\in \left[ \frac{1}{11}, \frac{1}{9}\right]$. Hence, $x_5$ satisfies \eqref{eq-pri-index-k}, and therefore $A=\{ x_6, \ldots, x_{10}\}$ and $B = \{ x_1, \ldots, x_4\}$, and every cone $N(x, A, B)$ then satisfies $|A| = 5$ and $|B|=4$. It then contains $5!\cdot 4! = 2880$ maximal comonotone cones, and therefore the number of all distinct MESCs is $\frac{10!}{5!\cdot 4!} = 1260$, which is then equal to the number of extreme points. Moreover, this is the maximal number of extreme points for a credal set corresponding to a coherent PRI model on 10 elements. 
\end{example}
\begin{example}
	Let this time $n = |\states| = 10$ and set $l(x) = \frac{1}{20}$  and $u(x) = \frac{1}{9}$ for every $x\in \states$. Given an ordering $x_1, \ldots, x_{10}$ of elements of $\states$, it turns out that $x_9$ is exactly the element satisfying \eqref{eq-pri-index-k}. Due to symmetry, we can conclude that every MESC in this case is of the form $N(x, A, B)$ where $|A|=1$ and $|B|=8$. By Proposition~\ref{prop-number-cones-PRI}, all of them contain exactly $8!=40320$ maximal comonotone cones. The number of cones must therefore be exactly $\frac{10!}{8!} = 90$, which coincides with the lower bound in Corollary~\ref{cor-number-cones-bounds}. 
\end{example}

\section{Conclusions}
Normal cones prove to be a useful tool for a better understanding of credal sets and numerical procedures related to them. The aim of this paper is to provide a comprehensive description of the structure of normal cones corresponding to credal sets of coherent lower probabilities. General properties introduced in the first part were then used to give a detailed description of the normal cone structure for two important families of imprecise probabilities, 2-monotone lower probabilities and probability intervals. 

The methods proposed in this paper will serve to complement and improve upon existing results using normal cone based methods. Models whose analysis has been shown to benefit from such an approach are computations related to imprecise stochastic processes, particularly those in continuous time. Another area within the theory of imprecise probabilities that remains largely unexplored is the analysis of the sensitivity of coherent lower probabilities to perturbations. Better understanding of the structure of the corresponding credal sets based on the approaches presented in this paper could help in such an analysis as part of our future research. Research on other important classes of lower probabilities, such as $p$-boxes and their multivariate generalizations, could also benefit from the approach proposed here.

	\section*{Acknowledgements}
	\begin{enumerate}
		\item The author acknowledges the financial support from the Slovenian Research Agency (research core funding No. P5-0168).
		\item The author is grateful to the two anonymous referees for careful readings of previous version of this paper and for many valuable suggestions.
	\end{enumerate}

\bibliography{references_all}
\end{document}